\documentclass{amsart}
\usepackage{graphicx} 
\usepackage{amsmath}
\usepackage{amsfonts}
\usepackage{amsthm}

\usepackage{amsmath,amsfonts,amsthm,mathrsfs,amssymb}
\usepackage[all]{xy}
\usepackage{enumerate}
\usepackage{graphicx,epsfig}
\usepackage[pagebackref]{hyperref}
\usepackage{mathtools}
\usepackage{bm}
\usepackage{tikz}
\usepackage{pgf}
\usepackage{epsfig}

\usetikzlibrary{matrix}

\theoremstyle{plain}
\newtheorem{thm}{Theorem}[section]
\newtheorem{cor}[thm]{Corollary}

\newtheorem{prop}[thm]{Proposition}

\theoremstyle{definition}
\newtheorem{defn}[thm]{Definition}

\newtheorem{rem}[thm]{Remark}

\newcommand{\TT}{\mathbb{T}}
\newcommand{\HH}{\mathbb H}

\newcommand{\pd}{\mathrm{pd}}
\newcommand\idd{\mathfrak{d}}

\newcommand\NN{\mathbb N}
\newcommand\SSS{\mathbb S}
\newcommand\PPP{\mathbb P}

\newcommand\im{\mathfrak{m}}

\newcommand\Tor{\textnormal{{Tor}}}

\newcommand\bq{{\underline{q}}}
\newcommand\bb{{\underline{\bullet}}}

\def\t{\theta}
\def\om{\omega}

\def\sC{{^\hbar{C}}}

\def\ZZ{{\mathbb Z}}

\def\C.{C_\bullet}
\def\F.{F_\bullet}
\def\k.{\mathcal{K}_{\bullet}}

\newcommand\intF{\Im}
\newcommand\compF{\textnormal{U}}

\newcommand\Face{\textnormal{F}}
\newcommand\ima{\textnormal{im}}

\newcommand\coker{\textnormal{coker}}

\newcommand\depth{\textnormal{depth}}

\newcommand\codim{\textnormal{codim}}

\newcommand\Supp{\textnormal{Supp}}

\newcommand\Ext{\textnormal{Ext}}

\newcommand\ext{\textnormal{Ext}}
\newcommand\Hom{\textnormal{Hom}}

\title[Multiple Tor modules]{Multiple Tor modules: rigidity and Mayer-Vietoris spectral sequences}

\author[A. Banerjee]{Arindam Banerjee}
\address{Department of Mathematics, IIT Kharagpur, Kharagpur, India}
\email{123.arindam@gmail.com}

\author[M. Chardin]{Marc Chardin}
\address{Institut de Mathématiques de Jussieu, CNRS \& Sorbonne Université, France}
\email{marc.chardin@imj-prg.fr}

\author[R. Holanda]{Rafael Holanda}
\address{Departamento de Matem\'atica, CCEN, Universidade Federal de Pernambuco, Recife, PE, 50740-560, Brazil}
\email{rafael.holanda@ufpe.br}

\date{\today}

\keywords{Tor module, rigidity, Tor independence, Mayer-Vietoris sequence}
\subjclass[2020]{Primary 13D02, 18G40; Secondary 13C12}

\begin{document}

\begin{abstract}
We extend some properties of a pair of ideals described in terms of Tor modules to any number of ideals, including the well-known rigidity property. Those extensions require the development of a homological theory for spectral sequences arising from multiple complexes. Out of this theory, two new complexes associated with quotients by sums and quotients by products of the given ideals emerge, and their homologies are related via the Tor-independence property. In the multigraded setting, we describe the support regions of Tor modules for quotients by sums and products of ideals generated by variables in terms of each other.

\end{abstract}

\maketitle

\section{Introduction}

Throughout the paper, by default, {\it ring} means a commutative and unitary ring.

  For a pair of ideals $I, J$ in a ring $R$, the study of the modules $\Tor_i^R(R/I, R/J)$ has been one of the very popular research topics. The module $\Tor_1^R(R/I, R/J)$ is isomorphic to $(I\cap J)/IJ$, while $\Tor_2^R (R/I, R/J)$ is isomorphic to the kernel of the natural map $I \otimes_R J\rightarrow IJ$. 
  
  The vanishing of $\Tor$ modules sheds light on various properties of the corresponding pair of ideals; for instance, if $R$ is a regular local ring, the vanishing all these for $i>0$ is equivalent to the vanishing of the first one ($i=1$), and if $R/(I+J)$ is Cohen-Macaulay (e.g. primary for the maximal ideal of $R$) then this first one vanishes if and only if both $R/I$ and $R/J$ are Cohen-Macaulay and the codimension of $I+J$ is the sum of the codimensions of $I$ and $J$ in $R$. Geometrically, it corresponds to a proper intersection of Cohen-Macaulay schemes \cite{Ser}. 
  
 Serre and Auslander \cite{Ser, Aus}  proved the rigidity property for Tor modules ($\Tor_i^R(R/I , R/J)=0 \Rightarrow \Tor_j^R(R/I , R/J)=0 , \forall j\geq i$) over unramified regular local rings; this was then completed by Lichtenbaum \cite{Lic} for arbitrary regular local rings.

We here extend several of the above-mentioned results to the case of more than two ideals, and also treat the case of modules that need not be cyclic. An incarnation of the module $\Tor_i^R (R/I, R/J)$ is as the $i$-th homology module of the tensor product double complex built from flat $R$-resolutions of $R/I$ and $R/J$. It naturally generalises to any number $n$ of ideals $I_1, \ldots, I_n$ as the $i$-th homology of the multicomplex obtained by taking the tensor product of flat resolutions of the quotients $R/I_j$ for $j=1,\ldots,n$. The definition and several properties of these modules figure in \cite{EGAIII}, which also presents these for complexes of sheaves.

We first present here a generalization of the Auslander-Serre result for an arbitrary finite collection of ideals (or of finitely generated modules), and its declination in geometric terms, together with some additional features that are only of interest for three or more modules.\\

\begin{thm}{\rm (Theorem \ref{rigtor} and Corollary \ref{A.8})} Assume that $R$ is a regular local ring containing a field
and $M_{1},\ldots ,M_{s}$ are non zero  finitely generated
  $R$-modules. Then,\smallskip

\begin{enumerate}[\rm (1)]
\item \begin{itemize}
\item[(a)] $\Tor_{i}^{R}(M_{1},\ldots ,M_{s})=0$ implies 
  $\Tor_{j}^{R}(M_{1},\ldots ,M_{s})=0$  for all $j\geq i$.
  
\item[(b)] For any $i$ and $t<s$, $\Tor_{i}^{R}(M_{1},\ldots ,M_{s})= 0$ implies 
  $\Tor_{i}^{R}(M_{1},\ldots ,M_{t})= 0$.\smallskip
\end{itemize}

\item The following
are equivalent,

\begin{enumerate}[\rm(a)]
\item  $\Tor_{1}^{R}(M_{1},\ldots ,M_{s})=0$ and 
$M_{1}\otimes_{R}\cdots\otimes_{R}M_{s}$ is Cohen-Macaulay.

\item The codimension of $M_{1}\otimes_{R}\cdots\otimes_{R}M_{s}$ is 
the sum of the projective dimensions of the $M_{i}$'s.

\item The intersection of the $M_{i}$'s is proper
and every $M_{i}$ is Cohen-Macaulay.
\end{enumerate}
\end{enumerate}
\end{thm}


 In the case of two ideals over an arbitrary ring, there is a nice description of the first two Tor modules, as recalled above, and an important sequence intimately linked to this description:
 $$
 \xymatrix{
0\ar[r]&R/IJ\ar^(.4){\theta}[r]&R/I\oplus R/J \ar[r]& R/(I+J)\ar[r]& 0\\
 }
 $$
 that is exact on the right ($\coker (\theta )=R/(I+J)=\Tor_0^R(R/I,R/J)$) and such that $\ker (\theta )=(I\cap J)/IJ=\Tor_1^R(R/I,R/J)$.

As an illustration, we will now describe the case of three ideals $I_1, I_2, I_3\subseteq R$. 
Then, to get an extension of the sequence above, consider the maps
$$
\xymatrix{
R/I_2I_3\oplus R/I_1I_3\oplus R/I_1I_2 \ar^(.55){\varphi}[r]&R/I_1\oplus R/I_2 \oplus R/I_3\\}
$$
$$
\xymatrix{
R/I_1\oplus R/I_2 \oplus R/I_3 \ar^(.35){\psi}[r]&R/(I_2+I_3)\oplus R/(I_1+I_3)\oplus R/(I_1+I_2) \\}
$$
where both $\varphi$ and $\psi$ are induced by the matrix
$
{\footnotesize
\left(
\begin{array}{ccc}
0&1&-1\\ -1&0&1\\ 1&-1&0\\
\end{array} \right)}
$. Our results show that 
$$
\coker (\varphi )\simeq 
\coker (\psi )\simeq R/(I_1+I_2+I_3)=\Tor_0^R(R/I_1,R/I_2,R/I_3).
$$
 Also, assuming $I_i\cap I_j=I_iI_j$ for any $i\neq j$, the natural maps $\kappa_\varphi: R/I_1I_2I_3\to\ker\varphi$ and $\kappa_\psi:R/I_1I_2I_3\to\ker\psi$, given by the matrix $(\begin{array}{ccc}1&1&1\end{array})$, satisfy:
\begin{enumerate}[\rm(1)]
    \item $\coker\kappa_\varphi \simeq\coker\kappa_\psi\simeq\Tor^R_1(R/I_1,R/I_2,R/I_3)$,
    \item there exist natural surjective maps $$s_\varphi:\Tor_2^R(R/I_1,R/I_2,R/I_3)\to\ker\kappa_\varphi\quad\mbox{and}\quad s_\psi:\Tor^R_2(R/I_1,R/I_2,R/I_3)\to\ker\kappa_\psi.$$ Furthermore, both $s_\varphi$ and $s_\psi$ are isomorphisms if $I_i\otimes_RI_j\xrightarrow{can}I_iI_j$ is an isomorphism for all $i\neq j$.
\end{enumerate}

In particular, if $I_i\otimes_RI_j= I_iI_j=I_i\cap I_j$ for all $i\neq j$, then $\ker (\kappa_\varphi )$ and $\ker (\kappa_\psi )$ are both isomorphic to $\Tor^R_2(R/I_1,R/I_2,R/I_3)$. 

The modules $\Tor^R_i(R/I_1, R/I_2, R/I_3)$ for $i=1,2$ thus emerge as constraints for $R/I_1I_2I_3$ to be the kernel of $\varphi$ and $\psi$.

The description of the first Tor module in the case of two ideals, as $(I\cap J)/IJ$, extends as follows to an arbitrary number of ideals (see Remark \ref{A.6}), however, with a less transparent interpretation if $n>2$:

\begin{prop}
Let $I_1,\ldots, I_n$ be ideals of $R$ and $M$ be the submodule of $I_{1}\oplus \cdots \oplus I_{n}$ of
tuples $x=(x_{1},\ldots ,x_{n})$ such that $x_{1}+\cdots +x_{n}=0$. The module
$M$ contains the submodule $P$ generated by the tuples $x$ such that
$x_{\ell}=0$ except for two indices $i$ and $j$ and $x_{i}=-x_{j}\in I_{i}I_{j}$ and
$$
\Tor_{1}^{R}(R/I_{1},\ldots ,R/I_{n})\simeq M/P.
$$
\end{prop}
The example of three ideals illustrates a more general pattern. If $I_1,\ldots, I_n$ are ideals of $R$, similarly to the maps $\varphi$ and $\psi$ above, we introduce complexes $\SSS^\bullet$ and $\PPP_\bullet$, where $\SSS^0=R/I_1\cdots I_n$ and $\PPP_0=0$ and for $p>0$, 
$$
\SSS^p =\bigoplus_{i_1<\cdots <i_{p}}R/(I_{i_1}+\cdots +I_{i_{p}})\; e_{i_1}\wedge \cdots \wedge e_{i_p}\quad\mbox{and}\quad\PPP_p =\bigoplus_{i_1<\cdots <i_{p}}R/I_{i_1}\cdots I_{i_{p}} \; e_{i_1}\wedge \cdots \wedge e_{i_p}.
$$

What allows us to investigate the (co)homology of such complexes is the homological counterpart of the theory developed in \cite{CHN} for cohomological multiple complexes; see Section \ref{multcomp}. 

These constructions provide a relation between the (co)homologies of $\SSS^\bullet$, of $\PPP_\bullet$, and the multiple Tor modules. 
Defining Tor-independence of ideals $I_1,\ldots ,I_n$ by the vanishing of $\Tor^R_{i}(R/I_1,\ldots,R/I_n)$ for every $i>0$ (see Definition \ref{torinddef} and Remark \ref{torinddefformod}), this notion then relates to the (co)homologies of $\PPP_\bullet$ and $\SSS^\bullet$:

\begin{thm}
If any strict subset of  $\{ I_1,\ldots, I_n\}$ is Tor-independent, then
\begin{enumerate}[\rm(1)]
    \item {\rm(Corollary \ref{homologyofSSS})} $\Tor^R_n(R/I_1,\cdots, R/I_n)\simeq\ker(I_1\otimes\cdots\otimes I_n\to I_1\cdots I_n)$ and $$H^i(\SSS^\bullet)\simeq\Tor^R_{n-i}(R/I_1,\ldots,R/I_n), \forall i\geq2.$$
    \item {\rm(Proposition \ref{homologyofPPP})} $H_i(\PPP_\bullet)\simeq \Tor^R_{i-1}(R/I_1,\ldots, R/I_n)$ for all $i\leq n$.
\\ In particular, if any subset of $\{ I_1,\ldots, I_n\}$ is Tor-independent, then:
\item {\rm(Corollary \ref{torind})} the following complex is exact
$$
\xymatrix{
0\ar[r]&R/I_1\cdots I_n\ar[r]&\SSS^{1}\ar[r]&\SSS^{2}\ar[r]&\cdots\ar[r]&\SSS^n \ar[r]&0\\}
$$
\item {\rm(Proposition \ref{homologyofPPP})} and the following complex is exact
$$
\xymatrix{
0\ar[r]&\PPP_{n}\ar[r]&\cdots\ar[r]&\PPP_2\ar[r]&\PPP_1\ar[r]&R/(I_1+\cdots+I_n)\ar[r] & 0.\\}
$$

\end{enumerate}
\end{thm}

These homological constructions also connect Tor modules relative to quotients by sums of these ideals to ones relative to quotients by products. Under Tor-independence conditions, it takes the form  of two spectral sequences:

\begin{prop}\label{twospectrals}
If any subset of $\{I_1,\ldots, I_n\}$ is Tor-independent, then for any $R$-module $M$, there exist two spectral sequences
\begin{enumerate}[\rm(1)]
    \item {\rm(Corollary \ref{spectraltorfromsumtoproduct})} 
$E^1_{p,q}=\oplus_{i_1<\cdots<i_p}\Tor^R_q(M, R/I_{i_1}+\cdots+I_{i_p})\Rightarrow\Tor^R_{q-p+1}(M, R/I_1\cdots I_n).$
\item {\rm(Corollary \ref{spectraltorfromproducttosum})}
$E^1_{p,q}=\oplus_{i_1<\cdots<i_p}\Tor^R_q(M, R/I_{i_1}\cdots I_{i_p})\Rightarrow\Tor^R_{q+p-1}(M, R/I_1+\cdots+I_n).$

\end{enumerate}
\end{prop}

Since ideals generated by sets of variables with empty two-by-two intersections in a polynomial ring (over any ring and with arbitrary grading) are Tor-independent, it allows us to apply the spectral sequences in Proposition \ref{twospectrals} to prove that Tor modules with respect to sums have the same support region (i.e., the non-vanishing degrees of their graded components) as Tor modules with respect to products, see Proposition \ref{supportoftors}. This fact has consequences in the study of extensions of the Castelnuovo-Mumford regularity to a multigraded setting and is in part at the origin of these investigations.



  


\section{Multiple Tor modules}\label{multipletormodules}

\begin{def}\label{defmtor} Let $M_{1},\ldots
,M_{s}$ be $R$-modules and $F_{\bullet}$ be  the tensor products 
over $R$ of the canonical free resolutions  of $M_{1},\ldots ,M_{s}$. Then 
$$
\Tor_{i}^{R}(M_{1},\ldots ,M_{s}):=H_{i}(F_{\bullet}).
$$
\end{def}

We note that the tensor product over the empty set (i.e., $s=0$) is $R$, while for $s=1$, $\Tor_0^R(M_1)=M_1$ and $\Tor_i^R(M_1)=0$ for $i>0$.

\begin{prop}\label{A.2} Let $M$ be an $R$-module and $F_{\bullet}$ be a 
complex of flat $R$-modules. If $L_{\bullet}$ is a free resolution of $M$, there is a 
natural isomorphism, 
$$
H_{i}(F_{\bullet}\otimes_{R}L_{\bullet})\simeq H_{i}(F_{\bullet}\otimes_{R}M).
$$
\end{prop}

\begin{proof} Consider the spectral sequence with
$E^{1}_{p,q}=H_{q}(F_{p}\otimes_{R}L_{\bullet})\simeq F_{p}\otimes_{R}H_{q}(L_{\bullet})$
and $E^{2}_{p,q}=H_{p}(E^{1}_{\bullet, q})$ that 
abuts to $H_{p}(F_{\bullet}\otimes_{R}L_{\bullet})$ and note that by
hypothesis $H_{q}(L_{\bullet})=0$ for $q\not= 0$ and
$H_{0}(L_{\bullet})=M$.\end{proof}

Notice that if $F_{\bullet}'$ is the tensor product of any choice of flat resolutions of  
$M_{1},\ldots ,M_{s}$, then $H_{i}(F_{\bullet}')\simeq\Tor_{i}^{R}(M_{1},\ldots ,M_{s})$. Also,

\begin{cor}\label{A.3} Let $M,M_{1},\ldots
,M_{s}$ be $R$-modules and $F_{\bullet}$ be  the tensor products 
over $R$ of free resolutions of $M_{1},\ldots ,M_{s}$. Then,
$$
\Tor_{i}^{R}(M_{1},\ldots ,M_{s},M)\simeq H_{i}(F_{\bullet}\otimes_{R}M).
$$
\end{cor}
In particular $\Tor_{i}^{R}(M_{1},\ldots ,M_{s},R)\simeq \Tor_{i}^{R}(M_{1},\ldots ,M_{s})$.

If $(R,\im )$ is local regular of dimension $n$, this result implies that 
$\Tor_{i}^{R}(M_{1},\ldots ,M_{s})=0$ for $i>n(s-1)$. Also notice in this case 
that if $M_{1}=\cdots =M_{s}=R/\im$, $\Tor_{n(s-1)}^{R}(M_{1},\ldots ,M_{s})\simeq
R/\im \not= 0$.
\medskip

\begin{prop}\label{A.4} Let $M_{1},\ldots
,M_{s},N_{1},\ldots ,N_{t}$ be $R$-modules. There exists a
spectral sequence
$$
E^{2}_{p,q}=\Tor_{p}^{R}(M_{1},\ldots ,M_{s},\Tor_{q}^{R}(N_{1},\ldots
,N_{t}))\ 
\Rightarrow \ \Tor_{p+q}^{R}(M_{1},\ldots ,M_{s},N_{1},\ldots
,N_{t}).
$$
\end{prop}

\begin{proof} Let $F_{\bullet}$ and $L_{\bullet}$ be  the tensor products 
over $R$ of free resolutions of $M_{1},\ldots ,M_{s}$ and $N_{1},\ldots ,N_{t}$, 
respectively. 

The double complex $C_\bullet$ with $C_{p,q}:=F_{p}\otimes_{R}L_{q}$ gives 
rise to two spectral sequences both abutting to $H_{\bullet}({\rm Tot}(C_\bullet))\simeq \Tor_{\bullet}^{R}(M_{1},\ldots ,M_{s}
,N_{1},\ldots,N_{t})$. One of them has as second page
$E^{2}_{p,q}=H_{p}(F_{\bullet}\otimes_{R}H_{q}(L_{\bullet}))$
and the result follows from Corollary \ref{A.3}.\end{proof}

Other spectral sequences for multiple Tor modules appear in  \cite[\S 6]{EGAIII}.

\begin{prop}\label{A.5} If $S\rightarrow R$ is a flat map and 
$M_{1},\ldots,M_{s}$ are $S$-modules, then
$$
\Tor_{i}^{R}(M_{1}\otimes_{S}R,\ldots ,M_{s}\otimes_{S}R)\simeq
\Tor_{i}^{S}(M_{1},\ldots ,M_{s})\otimes_{S}R
$$
for every $i$.
\end{prop}

\begin{proof} This is a special case of \cite[6.9.2]{EGAIII}. It follows by taking a free resolutions $F^{j}_\bullet$ of $M_j$ for every $j$  and noticing that $(F^{1}_\bullet\otimes_S\cdots\otimes_SF^{s}_\bullet)\otimes_SR\simeq(F^{1}_\bullet\otimes_SR)\otimes_R\cdots\otimes_R(F^{s}_\bullet\otimes_SR)$, where each $F^j_\bullet\otimes_SR$ is a flat $R$-resolution of $M_j\otimes_SR$.\end{proof}

\begin{rem}\label{A.6} If $I_{1},\ldots ,I_{s}$ are ideals in $R$, then
$\Tor_{1}^{R}(R/I_{1},\ldots ,R/I_{s})$ admits the following
description that generalizes the isomorphism
$\Tor_{1}^{R}(R/I,R/J)\simeq (I\cap J)/IJ$. 

Let $M$ be the submodule of $I_{1}\oplus \cdots \oplus I_{s}$ of
tuples $x=(x_{1},\ldots ,x_{s})$ such that $x_{1}+\cdots +x_{s}=0$. Then
$M$ contains the submodule $P$ generated by the tuples $x$ such that
$x_{\ell}=0$ except for two indices $i$ and $j$ and $x_{i}=-x_{j}\in I_{i}I_{j}$.
Then,
$$
\Tor_{1}^{R}(R/I_{1},\ldots ,R/I_{s})\simeq M/P.
$$

This follows from Remark \ref{descriptionofmultitors} with $p=1$, which also describes some higher Tor modules similarly, under some additional hypothesis.
\end{rem}

\begin{thm}\label{rigtor} Assume that $R$ is a regular local ring containing a field
and $M_{1},\ldots ,M_{s}$ are non zero finitely generated
  $R$-modules. Then,\smallskip

\begin{enumerate}[\rm (1)]
\item \begin{itemize}
\item[(a)] $\Tor_{i}^{R}(M_{1},\ldots ,M_{s})=0$ implies 
  $\Tor_{j}^{R}(M_{1},\ldots ,M_{s})=0$  for all $j\geq i$.\smallskip
  
\item[(b)] For any $i$ and $t<s$, $\Tor_{i}^{R}(M_{1},\ldots ,M_{s})= 0$ implies 
  $\Tor_{i}^{R}(M_{1},\ldots ,M_{t})= 0$.\smallskip
\end{itemize}

\item Let $j:=\max\{ i\ |\ \Tor_{i}^{R}(M_{1},\ldots ,M_{s})\not= 0\}$. Then
$$
\pd M_{1}+\cdots +\pd M_{s}=\dim R+j-\varepsilon ,
$$
with $0\leq \varepsilon \leq \dim \Tor_{j}^{R}(M_{1},\ldots ,M_{s})$. Let
$\varepsilon_{0}:=\min_{i}\{ \depth  \Tor_{j-i}^{R}(M_{1},\ldots ,M_{s})
+i\}$, then $\varepsilon \geq  \varepsilon_{0}$ and  equality holds if $\varepsilon_{0}=\depth  \Tor_{j}^{R}
(M_{1},\ldots ,M_{s})$.
\end{enumerate}
\end{thm}

\begin{proof} First we complete $R$ for the $\im$-adic filtration (where $\im$ is the 
maximal ideal of $R$), and use Cohen 
structure theorem to reduce to the case of a power series ring over $k:=R/\im$. Notice
that $R\rightarrow \hat R$ is flat and as the $M_{i}$'s are finite $\hat M_{i}=M_{i}
\otimes_{R}\hat R$, $\dim_{R}M_{i}=\dim_{\hat R}\hat M_{i}$ and  $\depth_{R} 
M_{i}=\depth_{\hat R}\hat M_{i}$ for every $i$. 

We let $n:=\dim R$, and may assume that $R=k[[X_{1},\ldots ,X_{n}]]$.

Consider $S$ the completed tensor product over $k$ of $s$ copies of $R$ (see \cite{Ser} for such tensor products). Notice
that  $S\simeq k[[X_{i,j}\ |\ 1\leq i\leq n,\ 1\leq j\leq s ]]$, if
we set $X_{i,j}:=1\otimes \cdots \otimes 1 \otimes
X_{i}\otimes 1\cdots \otimes 1$ where $X_{i}$ is at the $j$-th
position. The ring $S$ is local regular with maximal ideal $\im_{S}=(X_{i,j})=
\idd_{S}+\im /\idd_{S}$,  where $\idd_{S}$ defined by the exact sequence
$$
0\rightarrow \idd_{S}\rightarrow S\buildrel{\rm mult}\over\longrightarrow R\rightarrow 0.
$$
Notice that $\idd_{S}$ is generated by a regular sequence, for instance 
$f:=(X_{i,j}-X_{i,(j+1)}\ |\ 1\leq i\leq n,\
1\leq j<s )$ is a minimal generating $n(s-1)$-tuple. 
Considering $R$ as an $S$-module via the above exact sequence, for any tuple
$L_{1},\ldots ,L_{t}$ of free $R$-modules, there is a natural isomorphism
$$
L_{1}\otimes_{R}\cdots \otimes_{R}L_{t}\buildrel{\sim}\over\longrightarrow
(L_{1}\hat\otimes_{k}\cdots \hat\otimes_{k}L_{t})\otimes_{S}R.
$$

We put $P:=M_{1}\hat\otimes_{k}\cdots \hat\otimes_{k}M_{s}$, choose
a minimal free $R$-resolution $F^{(i)}_{\bullet}$  of $M_{i}$, set 
$F^{k}_{\bullet}:=F^{(1)}_{\bullet}\hat\otimes_{k}\cdots
\hat\otimes_{k}F^{(s)}_{\bullet}$ and $F^{R}_{\bullet}:=F^{(1)}_{\bullet}
\otimes_{R}\cdots\otimes_{R}F^{(s)}_{\bullet}$. 

Following \cite[V-11]{Ser}, notice that $F^{k}_{\bullet}$ is a minimal free $S$-resolution of $P$, the Koszul complex $K_{\bullet}(f;S)$ is a minimal free $S$-resolution of $R=S/\idd_{S}$,
and there are isomorphisms 
$$
\Tor_{i}^{R}(M_{1},\ldots ,M_{s})\simeq H_{i}(F^{R}_{\bullet})
     \simeq H_{i}(F^{k}_{\bullet}\otimes_{S}R)
     \simeq \Tor_{i}^{S}(P,R)
     \simeq H_{i}(K_{\bullet}(f;P)),
$$
so that (1)(a) follows by \cite[1.6.31]{BH} as $(S,\im_{S})$ is local Noetherian and 
$\idd_{S}\subset \im_{S}$. 

Item (1)(b) follows from \cite[1.6.18]{BH}, since setting $f':=(X_{i,j}-X_{i,(j+1)}\ |\ 1\leq i\leq n,\
1\leq j<t )$, $H_{i}(K_{\bullet}(f;P))=0$ implies $H_{i}(K_{\bullet}(f';P))=0$. But 
$$
H_{i}(K_{\bullet}(f';P))=\Tor_{i}^{R}(M_{1},\ldots ,M_{t})\hat\otimes_{k}M_{t+1}\hat\otimes_{k}\cdots \hat\otimes_{k}M_{s},
$$
which is not zero, unless $\Tor_{i}^{R}(M_{1},\ldots ,M_{t})=0$, since $M_{t+1},\ldots ,M_s$ are non zero. One can also prove (1)(b) by analyzing the corner of the spectral sequence in Proposition \ref{A.4}.

For (2) let $T^{\bullet}$ be the total complex associated to the 
double complex $\Hom_S(K_{\bullet}(f;P), I^\bullet)$, where $I^\bullet$ is an injective resolution of the canonical module $\omega_S$ of $S$, and $T^{i}:=\bigoplus_{p+q=i}
\Hom_S(K_{p}(f;P),I^q)$ and choose the upper index as row index. We will
estimate in two ways the number $\t :=\max \{ i\ |\ H^{i}(T^{\bullet})\not= 0\}$. The homology of $T^{\bullet}$ is the abutment of the two spectral sequences associated
to the horizontal and vertical filtrations of the double complex.


For the vertical filtration, one has $_{v}E^1_{p,q}=\Ext^q_S(K_p(f;P),\omega_S)
\simeq \Ext^q_S(P,\omega_S)^{{n(s-1)}\choose{p}}$.


Therefore  $_{v}E^1_{p,q}=0$
for $q>\pd_SP$ and for $p>n(s-1)$ which shows that $\t \geq \pd_SP+n(s-1)$. 
Also  ${_{v}E^\infty_{n(s-1),\pd_SP}}\simeq{}_{v}E^2_{n(s-1),\pd_SP}\simeq\ext^{\pd_SP}_{S}(P,\om_{S})/\idd_{S}\ext^{\pd_SP}_S(P,\om_{S})\neq0$ by Nakayama's
lemma. We deduce that $\t = \pd_SP+n(s-1)$.


The other spectral sequence has second terms  $_{h}E^2_{p,q}\simeq 
\Ext^q_S(\Tor^{R}_{p}(M_{1},\ldots ,M_{s}),\omega_S)$, and an easy computation 
gives  $ns-\varepsilon_{0}+j=\max\{ p+q\ |\ {}_{h}E^2_{p,q}
\not= 0\}$ so that $\t =\max\{ p+q\ |\ {}_{h}E^{\infty}_{p,q}
\not= 0\}=:ns-\varepsilon +j\leq ns-\varepsilon_{0}+j$ and it is clear from its definition 
that $ \varepsilon_{0}\geq 0$. Also if $\varepsilon_{0}=\depth \Tor^{R}_{j}
(M_{1},\ldots ,M_{s})$ 
we have ${_{h}E^{\infty}_{j, ns-\varepsilon_{0}}}\simeq{}_{h}E^2_{j, ns-\varepsilon_{0}}\neq0$ 
which implies that $\t \geq ns-\varepsilon_{0}+j$ and shows that $\varepsilon =\varepsilon_{0}$ in this case. We have noticed above that the projective dimension of $P$ (over $S$) is equal to the 
sum of the projective dimensions of the $M_{i}$'s (over $R$).
Therefore, $\pd_SP+n(s-1)=\t=ns-\varepsilon+j$ gives the desired formula. It remains to prove that $\varepsilon \leq d:=\dim  \Tor^{R}_{j}(M_{1},\ldots ,M_{s})$. 


According to \cite[Lemma 1.9]{Sch}, ${}_{h}E^{2}_{j, ns-d}=\ext^{ns-d}_{S}( \Tor^{R}_{j}(M_{1},\ldots ,M_{s}),\om_{S})$ is a module
of dimension $d$, and  ${}_{h}E^{\infty}_{j, ns-d}$ is isomorphic to a submodule of ${}_{h}E^{2}_{j, ns-d}$ that coincides with ${}_{h}E^{2}_{j, ns-d}$ in dimension $d-1$, because, for $\ell\geq 2$, ${}_{h}E^{\ell}_{j-\ell+1, ns-d+\ell}$ is isomorphic to a module 
supported in dimension at most $d-\ell$. \end{proof}

The next corollary extends \cite[V-20 Corollaire]{Ser}, for regular local rings containing a field. 

Recall that an intersection of equidimensional subschemes $V_i$ of an irreducible variety $V$ is proper if its codimension equals the sum of the codimension of the $V_i$'s. This extends to not necessarily cyclic finitely generated modules $M_i$ over a domain $R$, by requiring that $\sum_i \codim M_i=\codim M$, with $M$ the tensor product over $R$ of the $M_i$'s.

\begin{cor}\label{A.8}Assume that $R$ is a regular local ring containing a field and
$M_{1},\ldots ,M_{s}$ are non zero finitely generated $R$-modules. The following
are equivalent,

\begin{enumerate}[\rm(1)]
\item  $\Tor_{1}^{R}(M_{1},\ldots ,M_{s})=0$ and 
$M_{1}\otimes_{R}\cdots\otimes_{R}M_{s}$ is Cohen-Macaulay.

\item The codimension of $M_{1}\otimes_{R}\cdots\otimes_{R}M_{s}$ is 
the sum of the projective dimensions of the $M_{i}$'s.

\item The intersection of the $M_{i}$'s is proper
and every $M_{i}$ is Cohen-Macaulay.
\end{enumerate}
\end{cor}

 \begin{proof}  




Write $M:=M_1\otimes_R\cdots\otimes_RM_s$ and let $j, \varepsilon$, and $\varepsilon_0$ be as in Theorem \ref{rigtor}. Assume (1). In Theorem \ref{rigtor} one has $j=0$ and $\depth M=\varepsilon_0=\varepsilon$ which means that $M$ is Cohen-Macaulay, and $\sum_i\pd M_i=j+\dim R-\varepsilon=\dim R -\dim M=\codim M$, whence (2).

Observe that \cite[V-18 Théorème 3]{Ser} gives $\codim M_1+\cdots+\codim M_s\geq\codim M$ which is equivalent to
$$(s-1)\dim R-\sum_{i=1}^s\dim M_i+\dim M\geq0,$$
and by the Auslander-Buchsbaum formula, $j$ can be written as
\begin{equation}\tag{$\star$}\label{jformula}j=\sum_{i=1}^s(\dim M_i-\depth M_i)+((s-1)\dim R-\sum_{i-1}^s\dim M_i+\dim M)+(\varepsilon-\dim M).\end{equation}

If (2) holds, then $\dim R-\dim M=\codim M=\sum_i\pd M_i=j+\dim R-\varepsilon$, that is, $\dim M+j=\varepsilon$. Since $\varepsilon\leq\dim M$, we conclude that $j=0, \varepsilon=\dim M$ and $\depth M=\varepsilon_0=\varepsilon$ by Theorem \ref{rigtor} again. It proves (1). Furthermore, applying this to (\ref{jformula}) and noticing that each component therein is non-negative, one concludes (3).

The hypotheses in (3) mean that (\ref{jformula}) has the form $j=\varepsilon-\dim M$, which implies again that $j=0$ and $M$ is Cohen-Macaulay, hence (1).

\end{proof}

\section{Spectral sequences from homological multicomplexes}\label{multcomp}

\subsection{Setup}

In this section, we develop the homological counterpart of results in \cite{CHN} for cohomology. To do so, the same notation and terminology therein are adopted, while definitions and illustrations are avoided to prevent repetition; see also \cite{Ver} for a general theory on multicomplexes.

Let $C_{\underline{\bullet}}$ be a (homological commuting) $n$-multicomplex of $R$-modules with components $C_\bq$ verifying $C_\bq=0$ for $\bq\not\in \NN^n=\oplus_{i=1}^{n}\NN e_i$,  with $e_i$ being the i-th canonical basis element of $\mathbb Z^n$. For such a multicomplex, as in \cite{CHN}, we will consider several complexes attached to the faces of the polyhedral cone ${\mathbb R}_{\geq 0}^n$. We recall where they come from. For $i_1<\cdots <i_p$, let 

-- $\Face_{i_1,\ldots ,i_p}:=\NN e_{i_1}\oplus \cdots \oplus \NN e_{i_p}$, denote a $p$-dimensional face, 



-- $\intF_{i_1,\ldots ,i_p}:=\{ \bq\in \NN^n\ \vert \ q_i= 0\Leftrightarrow i\not\in \{ i_1,\ldots ,i_p\}\}$,  the interior of this $p$-dimensional face,


-- $\compF_{i_1,\ldots ,i_p} :=\NN^n\setminus \Face_{i_1,\ldots ,i_p} $,  the complement of this $p$-dimensional face.



These sets could as well be defined in terms of coordinates that are zero on these faces, namely $\Face_{i_1,\ldots ,i_p}^*:=\Face_{\{ 1,\ldots ,n\}\setminus \{ i_1,\ldots ,i_p\} }$ and similarly for the three other sets. Notice that $\Face_{1,\ldots ,n}=\Face^*_\emptyset =\NN^n$ and  $\Face_\emptyset =\Face^*_{1,\ldots ,n}=\underline 0$.

Write $C_{\bb}^{F_{i_1,\ldots ,i_p}}$ for the subcomplex of $C_{\bb}$ obtained by replacing $C_\bq$ by the zero module unless $\bq\in F_{i_1,\ldots ,i_p}$. Similarly, we define $C_{\bb}^{\compF_{i_1,\ldots ,i_p}}$ and notice that $C_\bb^{\compF_{i_1,\ldots,i_p}}=C_\bb/C_\bb^{\Face_{i_1,\ldots,i_p}}$ is a quotient of $C_\bb$. We denote by
$C_{\bb}^{\intF_{i_1,\ldots ,i_p}}$ the quotient of $C_{\bb}^{\Face_{i_1,\ldots ,i_p}}$ mapping to zero the modules for 
$\bq\not\in \intF_{i_1,\ldots ,i_p}$.

\subsection{The hypercube augmentation}\label{hypercubesection}

Let $T_\bb $ be the trivial hypercube commuting $n$-multicomplex on $C_{\underline 0}$: $T_{\underline q}=C_{\underline 0}$ if $\bq\in \{-1,0\} ^n$  and $0$ otherwise, and differentials are the identity if source and target are in degrees that belong to $\{-1,0\} ^n$ and $0$ else. If $C_\bb$ is a $n$-multicomplex, we define a map from $C_\bb$ to $T_\bb $ by 
$$
\psi_{i_1,\ldots ,i_p}\ :\ x\in C_{e_{i_1}+\cdots +e_{i_p}}\mapsto d_{e_{i_{1}},i_1}\circ \; d_{e_{i_{1}}+e_{i_{2}},i_{2}}\; \circ \;\cdots \;\circ \; d_{e_{i_1}+\cdots +e_{i_{p}},i_{p} }(x)\in T_{e_{i_1}+\cdots +e_{i_p}}(C_{\underline 0})=C_{\underline 0}
$$
and notice that it provides a commuting $(n+1)$-multicomplex $\sC_{ \bb, \bullet}$ with non zero components  sitting in degrees that belong to $\NN^n \times \{-1,0\}$, 
with $\sC_{\bb ,-1}=T_\bb$ and $\sC_{\bb ,0}=C_\bb$. In other words, $T_\bb $ is the $n$-multicomplex $(K_\bullet (1;R)\otimes \cdots \otimes K_\bullet (1;R))\otimes C_{\underline 0}$ ($n$ iterations) with $T_{\underline q}=(\bigotimes_i K_{q_i}(1;R))\otimes C_{\underline 0}$, which is $C_{\underline 0}$ for $\underline0\leq \bq\leq \underline{1}$ and zero else. The map $\psi$ arises from the maps:
$$
\xymatrix{
C_{\underline0}\ar^{1}[d]&C_{e_i}\ar^{d_{e_i,i}}[l]\ar^{d_{e_i,i}}[d]\\
C_{\underline0}=T_{\underline0}&T_{\underline1}=C_{\underline0}\ar^{1}[l]
}
$$
Denote by $\sC_\bullet$ and $C_\bullet$ the totalization of $\sC_{\bb,\bullet}$ and $C_\bb$, respectively. Note also that from the exact sequence of $(n+1)$-multicomplexes
$$0\to T_{\underline\bullet}\to\sC_{\bb, \bullet}\to C_{\underline\bullet}\to0$$
one gets $H_\bullet(\sC_\bullet)\simeq H_\bullet(C_\bullet)$, as $T_\bb$ has trivial homology (i.e., the homology of its totalization is trivial).

We finally define an augmented version ${^+{C}}_{\bb}^{\intF_{i_1,\ldots ,i_p}}$ of $C_{\bb}^{\intF_{i_1,\ldots ,i_p}}$ for $p>0$ by adding the module $C_{\underline 0}$ in homological degree $e_{i_1}+\cdots +e_{i_{p-1}}$ and adding the map
$$
\xymatrix{
C_{e_{i_1}+\cdots +e_{i_{p}}}\ar^(.45){d_{e_{i_{1}},i_1}\circ \; d_{e_{i_{1}}+e_{i_{2}},i_{2}}\; \circ \;\cdots \;\circ \; d_{e_{i_1}+\cdots +e_{i_{p}},i_{p} }}[rrrrrr]&&&&&&
C_{\underline 0}[-e_{i_1}-\cdots -e_{i_{p-1}}].\\}
$$
Notice that  $C_{\bb}^{\intF_{i_1,\ldots ,i_p}}$ starts in total homological degree $p$ with unique summand $C_{e_{i_1}+\cdots +e_{i_{p}}}$.






The totalizations of all these complexes obtained from $C_{\bb}$ will be denoted respectively by $C_{\bullet}^{\Face_{i_1,\ldots ,i_p} }$, 
$C_{\bullet}^{\intF_{i_1,\ldots ,i_p}}$ and ${^+{C}}_ {\bullet}^{\intF_{i_1,\ldots ,i_p}}$. These start respectively in homological degree 0 or higher,
$p$ or higher, and $p-1$ or higher. As the sets $\Face^*_{i_1,\ldots ,i_p}$, 
$\intF^*_{i_1,\ldots ,i_p}$ or $\compF^*_{i_1,\ldots ,i_p}$ are respectively equal to $\Face_{j_1,\ldots ,j_{n-p}}$, 
$\intF_{j_1,\ldots ,j_{n-p}}$ or $\compF_{j_1,\ldots ,j_{n-p}}$ for $ \{ j_1,\ldots ,j_{n-p}\}=\{ 1,\ldots ,n\}\setminus \{ i_1,\ldots ,i_p\} $, we may alternatively use this other notation.

Finally, we set 
${\intF}:= {\intF}_{1,\ldots ,n}$.


\subsection{Homological spectral sequences arising from multicomplexes}  

The next result is the homological version of \cite[Theorem 2.3]{CHN}. The existence of these spectral sequences will follow from natural filtrations on simple explicit constructions from $C_{\bb}$, or on $C_{\bb}$ itself. The rest of this section is devoted to detailing these constructions.

\begin{thm}\label{fourspec}
Let $C_{\underline{\bullet}}$ be a $n$-multicomplex satisfying $C_\bq=0$ for $\bq\not\in \NN^n=\oplus_{i=1}^{n}\NN e_i$. Then there exist four convergent spectral sequences as follows:

\begin{itemize}
\item[(1)] $E_{p,q}^1=\oplus_{i_1<\cdots <i_p}H_q(C_{\bullet}^{\Face^*_{i_1,\ldots ,i_p}})\Rightarrow H_{p+q}(C_{\bullet}^{\intF})$,

\item[(2)] $E_{p,q}^1=\oplus_{\substack{i_1<\cdots <i_p \\ p\not= n}}H_q(C_{\bullet}^{\Face^*_{i_1,\ldots ,i_p}})\Rightarrow H_{p+q}({^+{C}}_{\bullet}^{\intF})$,

\item[(3)] $E^1_{p,q}=\oplus_{i_1<\ldots<i_{p}}H_{p+q}(C_{\bullet}^{\intF_{i_1,\ldots ,i_p} })\Rightarrow H_{p+q}(C_{\bullet})$,

\item[(4)] $E^1_{p,q}=\oplus_{i_1<\ldots<i_{p}}H_{p+q}({^+{C}}_{\bullet}^{\intF_{i_1,\ldots ,i_p}})\Rightarrow H_{p+q}(C_{\bullet})$.
\end{itemize}

\end{thm}


We start with a construction that will be used for the first two spectral sequences.

Let $K_\bullet(1,\ldots,1;\_)$ stand for the Koszul complex of the sequence $1,\ldots,1$ (n times) -- such a Koszul complex is exact, since $n\geq 1$, by \cite[1.6.5(c)]{BH}. Write $C_{p,\bb}^{\Bbbk}$ for the subcomplex of the $(n+1)$-multicomplex  $K_\bullet (1,\ldots ,1;C_{\bb})$, with components the subcomplexes 
$$
C_{p,\bb}^{\Bbbk}:=\bigoplus_{i_1,\ldots ,i_p}C_{\bb}^{F_{i_1,\ldots ,i_p}^*}e_{i_1}\wedge\cdots \wedge e_{i_p}\subseteq K_p({{\underbrace{1,\ldots ,1}_{n\ {\rm times}}}};C_{\bb})=\bigoplus_{i_1,\ldots ,i_p}C_{\bb}\; e_{i_1}\wedge\cdots \wedge e_{i_p}.
$$



\textbf{Illustration for the case of double complexes:}

\begin{center}
$
\begin{matrix}
{\xymatrix@R=0pt@C=0pt{
\vdots&\vdots&\vdots&\vdots&\vdots&\\
\bullet&\bullet&\bullet&\bullet&\bullet&\cdots\\
\bullet&\bullet&\bullet&\bullet&\bullet&\cdots\\
\bullet&\bullet&\bullet&\bullet&\bullet&\cdots\\
\bullet&\bullet&\bullet&\bullet&\bullet&\cdots\\
\bullet&\bullet&\bullet&\bullet&\bullet&\cdots\\}}\\
C_{0,\bullet\bullet}^{\Bbbk}=C_{\bullet\bullet}\\
\end{matrix}
$
\ \ \ 
$
\begin{matrix}
{\xymatrix@R=0pt@C=0pt{
\vdots&\vdots&\vdots&\vdots&\vdots&\\
\circ&\circ&\circ&\circ&\circ&\cdots\\
\circ&\circ&\circ&\circ&\circ&\cdots\\
\circ&\circ&\circ&\circ&\circ&\cdots\\
\circ&\circ&\circ&\circ&\circ&\cdots\\
\bullet&\bullet&\bullet&\bullet&\bullet&\cdots\\}}\\
C_{1,\bullet\bullet}^{\Bbbk}\ \ [e_2]\\
\end{matrix}
$
\ \ \ 
$
\begin{matrix}
{\xymatrix@R=0pt@C=0pt{
\vdots&\vdots&\vdots&\vdots&\vdots&\\
\bullet&\circ&\circ&\circ&\circ&\cdots\\
\bullet&\circ&\circ&\circ&\circ&\cdots\\
\bullet&\circ&\circ&\circ&\circ&\cdots\\
\bullet&\circ&\circ&\circ&\circ&\cdots\\
\bullet&\circ&\circ&\circ&\circ&\cdots\\}}\\
C_{1,\bullet\bullet}^{\Bbbk}\ \ [e_1]\\
\end{matrix}
$
\ \ \ 
$
\begin{matrix}
{\xymatrix@R=0pt@C=0pt{
\vdots&\vdots&\vdots&\vdots&\vdots&\\
\circ&\circ&\circ&\circ&\circ&\cdots\\
\circ&\circ&\circ&\circ&\circ&\cdots\\
\circ&\circ&\circ&\circ&\circ&\cdots\\
\circ&\circ&\circ&\circ&\circ&\cdots\\
\bullet&\circ&\circ&\circ&\circ&\cdots\\}}\\
C_{2,\bullet\bullet}^{\Bbbk}\ \ [e_1\wedge e_2]\\
\end{matrix}
$
\\ 
\end{center}
Directions of maps are: down for vertical maps and to the left for horizontal ones.

\begin{prop}\label{mainspectral} With notations as above, the following holds:
\begin{itemize}
\item [(1)] For any $\bq\in \NN^n$, $H_p (C_{\bullet ,\bq}^{\Bbbk})=0$ for $p\not=0$ 
and 
$
H_0 (C_{\bullet ,\bq}^{\Bbbk})
=C_{\bq}^{\intF}.
$

\item [(2)] Let $C^{\Bbbk}_\bullet$ be the totalization of $C_{\bullet,\bb}^{\Bbbk}$, then 
$
H_i (C_\bullet^{\Bbbk})\simeq H_{i}(C_{\bullet}^{\intF}), \forall i.
$

\item [(3)] There is a spectral sequence,
$$
E_{p,q}^1=\oplus_{i_1<\cdots <i_p}H_q(C_{\bullet}^{\Face^*_{i_1,\ldots ,i_p}})\Rightarrow H_{p+q}(C_{\bullet}^{\intF} ).
$$
\end{itemize}
\end{prop}

\begin{proof}
Item (1) is equivalent to the exactness of $C_{\bullet ,\bq}^{\Bbbk}$ for $\bq\not\in \intF$, because if $\bq\in \intF$ then $C_{p ,\bq}^{\Bbbk}=0$ unless $p=0$ and $C_{0 ,\bq}^{\Bbbk}=C_{\bq}$. Assume that $\bq$ has exactly $t\geq 1$ coordinates equal to zero and $\bq \in \Face^*_{j_1,\ldots ,j_t}$. Then $C_{\bullet ,\bq}^{\Bbbk}$ is the exact subcomplex 
$$
K_p({{\underbrace{1,\ldots ,1}_{t\ {\rm times}}}};C_{\bq})\subseteq K_p({{\underbrace{1,\ldots ,1}_{n\ {\rm times}}}};C_{\bq})
$$
that corresponds to summands indexed by $e_{i_1}\wedge\cdots \wedge e_{i_p}$ for $\{ i_1,\ldots ,i_p \}\subseteq \{ j_1,\ldots ,j_t \}$.

For (2), denote by $C_{\bullet ,\bullet}^{\Bbbk}$ the double complex obtained by totalizing along $\NN^n$ the complex $C_{\bullet,\bb}^{\Bbbk}$. By (1), for any $q\in \NN$, $H_p(C_{\bullet ,q}^{\Bbbk})=0$ for $p\not= 0$ and $H_0 (C_{\bullet ,q}^{\Bbbk})=C_{q}^{\intF}$; hence $C_\bullet^{\Bbbk}$ is quasi-isomorphic to $C_{\bullet}^{\intF}$ and the conclusion follows.

For (3)  recall that $C_{p,\bb}^{\Bbbk}=\bigoplus_{i_1,\ldots ,i_p}C_{\bb}^{\Face^*_{i_1,\ldots ,i_p}}e_{i_1}\wedge\cdots \wedge e_{i_p}$, hence the second spectral sequence for $C_{\bullet ,\bullet}^{\Bbbk}$ has first terms $H_{q}(C_{p,\bullet}^{\Bbbk})=\oplus_{i_1<\cdots <i_p}H_q(C_{\bullet}^{\Face^*_{i_1,\ldots ,i_p}})$ and abuts to the homology of $C_\bullet^{\Bbbk}$ that is in turn  isomorphic to the one of $C_{\bullet}^{\intF}$ by (2).
\end{proof}

The second spectral sequence in Theorem \ref{fourspec} is a variant of the spectral sequence in Proposition \ref{mainspectral}(3).

\begin{prop}\label{augmentedmainspectral}
With notations as in Proposition \ref{mainspectral}, there is a spectral sequence 
$$
E_{p,q}^1=\oplus_{i_1<\cdots <i_p}H_q(C_{\bullet}^{\Face^*_{i_1,\ldots ,i_p}})\Rightarrow H_{p+q}({^+{C}}_\bullet^\intF)
$$
with $p$ in the range $0\leq p<n$ (i.e. $E_{p,q}^1=0$ for any $p\geq n$ and $q$).
\end{prop}

\begin{proof} Consider the complexes
$$
T_{p,\bb}^{\Bbbk}:=\bigoplus_{i_1,\ldots ,i_p}T_{\bb}^{\Face_{i_1,\ldots ,i_p}^*}e_{i_1}\wedge\cdots \wedge e_{i_p}\subseteq K_p({{\underbrace{1,\ldots ,1}_{n\ {\rm times}}}};T_{\bb})=\bigoplus_{i_1,\ldots ,i_p}T_{\bb}\; e_{i_1}\wedge\cdots \wedge e_{i_p},
$$
and $\sC^{\Bbbk}_{\bullet,\bb ,\bullet}$ with 
$$
\sC^{\Bbbk}_{\bullet,\bb ,0}=C_{\bullet,\bb}^{\Bbbk}\quad\quad \hbox{and}\quad\quad\sC^{\Bbbk}_{\bullet,\bb ,-1}=T_{\bullet,\bb}^{\Bbbk}.
$$
Proposition \ref{mainspectral}  shows that the total homology of $\sC^{\Bbbk}_{\bullet,\bb ,\bullet}$ is the one of ${^+{C}}_\bullet^\intF$, since $T_\bb^\intF$ has only non zero component $T_{\underline{\mathbf 1}}^\intF=C_{\underline0}$ and ${^+{C}}_\bullet^\intF$ is the mapping cone of $C_\bullet^\intF \rightarrow T_\bullet^\intF $, where the map from $C_{\underline{\mathbf 1}}^\intF$ to $T_{\underline{\mathbf 1}}^\intF =C_{\underline0}$ is $d_{e_{1},1}\circ \; d_{e_{1}+e_{2},2}\; \circ \;\cdots \;\circ \; d_{e_{1}+\cdots +e_{n},n }$, as described in Subsection \ref{hypercubesection}.

Now $T_{\bullet}^{\Face_{i_1,\ldots ,i_p}^*}$ has trivial homology unless $p=n$, showing that in this case,  
the totalization of $\sC^{\Bbbk}_{p,\bb ,\bullet}$, has the same homology as $C_{p,\bullet}^{\Bbbk}$. For $p=n$, $\sC^{\Bbbk}_{n,\bq ,\bullet}=0$ unless $\bq =0$
and the non zero part of  $\sC^{\Bbbk}_{n,\bb ,\bullet}$ is $\xymatrix{\sC^{\Bbbk}_{n,{\underline0} ,0}=C_{\underline0}\ar^{1}[r]&C_{\underline0}=\sC^{\Bbbk}_{n,{\underline0} ,-1}\\}$, with homology zero.
\end{proof}

We apply another construction to the main theorem's third and fourth spectral sequences. In the next propositions, we consider filtrations of a $n$-multicomplex and its hypercube augmentation.

\begin{prop}\label{ss3}Let $C_{\underline{\bullet}}$ be a $n$-multicomplex satisfying $C_\bq=0$ for $\bq\not\in \NN^n=\oplus_{i=1}^{n}\NN e_i$. Then there exists a convergent spectral sequence:
 $$E^1_{p,q}=\oplus_{i_1<\ldots<i_p}H_{p+q}(C_{\bullet}^{\intF_{i_1,\ldots ,i_p} })\Rightarrow H_{p+q}(C_{\bullet}).$$
\end{prop}

\begin{proof}
For $p\geq0$, define
$$X_p=\{(q_1,\ldots,q_n)\in\mathbb N^n: \mbox{at least} \ n-p \ \mbox{of the} \ q_j\mbox{'s are zero}\}.$$
The family of subcomplexes $F^p_\bullet$ given by $F^p=\oplus_{\underline{q}\in X_p}C_{\underline{q}}$ is a limited ascending filtration of $F^n=C_\bullet$ and therefore it yields a spectral sequence that abuts to $H_{p+q}(C_\bullet)$. As for $p\geq1$, $F^p/F^{p-1}\simeq\bigoplus_{i_1<\cdots<i_p}C_\bullet^{\intF_{i_1,\ldots,i_p}}$,
$$E^1_{p,q}=\oplus_{i_1<\ldots<i_p}H_{p+q}(C_{\bullet}^{\intF_{i_1,\ldots ,i_p} }).$$
\end{proof}

We also control the total homology of $C_{\bb}$ in terms of hypercube augmentations:

\begin{prop}\label{fromaugmentedinteriorspectral}Let $C_{\underline{\bullet}}$ be a $n$-multicomplex satisfying $C_\bq=0$ for $\bq\not\in \NN^n=\oplus_{i=1}^{n}\NN e_i$. Then there exists a convergent spectral sequence:

$$E^1_{p,q}=\bigoplus_{i_1<\cdots<i_p}H_{p+q}(^+{C}_\bullet^{\intF_{i_1,\ldots,i_p}})\Rightarrow H_{p+q}(C_\bullet).$$
\end{prop}

\begin{proof}
For $p\geq0$, define
$$X_p=\{(q_1,\ldots,q_n)\in\mathbb N^n: \mbox{at least} \ n-p \ \mbox{of the} \ q_j\mbox{'s are zero}\}$$
and $X'_p=X_p\times\ZZ$. Note that
$X_p=X_{p-1}\cup(\cup_{i_1<\cdots<i_p}\intF_{i_1,\ldots,i_p})$. The family of subcomplexes $F^p_\bullet$ given by $F^p=\oplus_{\underline{q}'\in X'_p}\sC_{\underline{q}'}$ is a limited ascending filtration of $F^n_\bullet=\sC_\bullet$ and therefore it yields a spectral sequence that abuts to $H_{p+q}(\sC_\bullet)\simeq H_{p+q}(C_\bullet)$.

We compute the first page of this spectral sequence by noticing that
$$F^p/F^{p-1}\simeq\bigoplus_{i_1<\cdots<i_p}{}^+C_\bullet^{\intF_{i_1,\ldots,i_p}}$$
for $p\geq1$ and $F^0:0\to C_{\underline0}\xrightarrow{1}C_{\underline0}\to0$.
\end{proof}



\section{Applications to Tor modules}\label{applicationssection}
\medskip

In this section, we consider $I_1, \ldots, I_n$ ideals in a ring $R$.

\subsection{Complexes associated to sums and products of ideals}\label{complexessection} 
\hfill

Consider the Koszul complexes $K_\bullet=K_\bullet (1,\ldots ,1;R)$ and $K^\bullet=K^\bullet (1,\ldots ,1;R)$ (1 appearing $n$ times) and for $p>0$,
$$
\tilde\SSS^p :=\bigoplus_{i_1<\cdots <i_{p}}(I_{i_1}+\cdots +I_{i_{p}})\; e_{i_1}\wedge \cdots \wedge e_{i_p}
\subseteq \bigoplus_{i_1<\cdots <i_{p}} R\;  e_{i_1}\wedge \cdots \wedge e_{i_p}=K^p
$$
 and 
$$
\tilde\PPP_p :=\bigoplus_{i_1<\cdots <i_{p}}I_{i_1}\cdots I_{i_{p}}e_{i_1}\wedge \cdots \wedge e_{i_p}
\subseteq \bigoplus_{i_1<\cdots <i_{p}} R\; e_{i_1}\wedge \cdots \wedge e_{i_p}=K_p.
$$

For $\tilde\SSS^0$ and $\tilde\PPP_0 $ there are two possible options for each:

--- $\tilde\SSS^0 =I_1\cdots I_n$ or $\tilde\SSS^0 =I_1\cap \cdots \cap I_n$ and

--- $\tilde\PPP_0=R $ or $\tilde\PPP_0 = I_1+\cdots +I_n$.

\noindent We choose the first options and consider the complexes 

$$
\tilde\SSS^\bullet : \xymatrix{
0\ar[r]&\tilde\SSS^0\ar[r]&\tilde\SSS^{1}\ar[r]&\tilde\SSS^{2}\ar[r]&\cdots\ar[r]&\tilde\SSS^n \ar[r]&0\\}
$$
and
$$
\tilde\PPP_\bullet : \xymatrix{
0\ar[r]&\tilde\PPP_n\ar[r]&\cdots\ar[r]&\tilde\PPP_{2}\ar[r]&\tilde\PPP_{1}\ar[r]&\tilde\PPP_0 \ar[r]&0\\}
$$
and set $\SSS^p :=K^p/\tilde\SSS^p$ and  $\PPP_p :=K_p/\tilde\PPP_p$. Thus $\SSS^0=R/I_1\cdots I_n$ and $\PPP_0=0$ and for $p>0$, 
$$
\SSS^p =\bigoplus_{i_1<\cdots <i_{p}}R/(I_{i_1}+\cdots +I_{i_{p}})\; e_{i_1}\wedge \cdots \wedge e_{i_p}\quad\mbox{and}\quad\PPP_p =\bigoplus_{i_1<\cdots <i_{p}}R/I_{i_1}\cdots I_{i_{p}} \; e_{i_1}\wedge \cdots \wedge e_{i_p}.
$$ 
We then have the exact sequences of complexes
$$
\xymatrix{
0\ar[r]&\tilde\SSS^\bullet \ar[r]&
K^\bullet (1,\ldots ,1;R) \ar[r]&
\SSS^\bullet \ar[r]
&0}
$$
and
$$
\xymatrix{0\ar[r]&\tilde\PPP_\bullet \ar[r]&K_\bullet (1,\ldots ,1;R)\ar[r]&\PPP_\bullet\ar[r]&0\\}
$$
which show that $H^i(\SSS^\bullet )=H^{i+1}(\tilde\SSS^\bullet )$ and  $H_i(\PPP_\bullet )=H_{i-1}(\tilde\PPP_\bullet )$, for any $i$. In particular, $H_1 (\PPP_\bullet )=R/(I_1+\cdots +I_n )$.

\subsection{Mayer-Vietoris spectral sequences for Tor modules}

\subsubsection{The spectral sequence Proposition \ref{augmentedmainspectral}}

Let $F_\bullet^{j}$, with $F_0^{j}=R$, be a free resolution of $R/I_j$ for $1\leq j\leq n$; then $F_\bb :=F_\bullet^{1}\otimes\cdots\otimes F_\bullet^{n}$ is a $n$-multicomplex. Given an $R$-module $M$ and $1\leq i_1<\cdots<i_p\leq n$, we consider
$$\Tor^R_j(M,R/I_{i_1},\ldots, R/I_{i_p})=H_j(M\otimes F_\bullet^{{i_1}}\otimes\cdots\otimes F_\bullet^{{i_p}})$$
as in Corollary \ref{A.3}. In particular, $H_i(F_\bullet )=\Tor^R_i(R/I_{1},\ldots, R/I_{n})$, for any $i$.

\newcommand\HHH{\mathbb H}


 Notice further that $H_i ({^+F}_\bullet^\intF[-n+1])=\Tor_{i-1}^R (I_1,\ldots ,I_n)$, for $i\geq2$, while $H_1 ({^+F}_\bullet^\intF[-n+1])=\ker (I_1\otimes\cdots\otimes I_n \rightarrow I_1\cdots I_n )$ and $H_0 ({^+F}_\bullet^\intF[-n+1])=R/I_1\cdots I_n$.\medskip

\begin{rem}\label{extensiontomodulesrmk}
More generally, if $F_\bullet^{M_j}$ is a free resolution of a module $M_j$ for $1\leq j\leq n$, then $F_\bb :=F_\bullet^{M_1}\otimes\cdots\otimes F_\bullet^{M_n}$ is a $n$-multicomplex and the modules $H_i (F_\bullet )=\Tor_i^R (M_1,\ldots ,M_n)$ and $H_i (P_\bullet)$, with $P_\bullet :={^+F}_\bullet^\intF[n-1]$ are both independent of the chosen resolutions, up to isomorphism. To see this, notice that if $G_\bullet^{M_j}$ is a free resolution of a module $M_j$, then $G_\bullet^{M_j}$ is homotopy equivalent to $F_\bullet^{M_j}$ for every $j$, and $F_\bullet^{M_{i_1}}\otimes\cdots\otimes F_\bullet^{M_{i_p}}$ is homotopy equivalent to $G_\bullet^{M_{i_1}}\otimes\cdots\otimes G_\bullet^{M_{i_p}}$, with homotopies given by the tensor product of the given choice of homotopies \cite{CE}. It follows that the homotopy of multicomplexes \cite[2.5, p. 60]{Ver} given by these homotopies induces a canonical isomorphism (i.e. independent of the chosen homotopies) $H_i (F_\bullet )\simeq H_i (G_\bullet )$ and also isomorphisms on all other terms in the first page of the spectral sequence of Proposition \ref{augmentedmainspectral}: $H_q(F_{\bullet}^{\Face^*_{i_1,\ldots ,i_p}})\simeq H_q(G_{\bullet}^{\Face^*_{i_1,\ldots ,i_p}})$, which in turn proves that the abutments coincide, this is the homology of $P_\bullet$.

Set $C_i:=\coker (d_2^{M_i} )$. The exact sequences $\xymatrix{0\ar[r]&C_i\ar[r]&F_0^{M_i}\ar[r]&M_i\ar[r]&0\\}$ and the diagram where paths following solid arrows are exact
$$
\xymatrix@R=8pt{
P_2\ar[r]&P_1\ar[dr]\ar@{-->}[rr]&&F_0^{M_1}\otimes\cdots \otimes F_0^{M_n}\ar[r]&H_0 (P_\bullet )\ar[r]&0\\
&&C_1\otimes\cdots\otimes C_n\ar[ur]\ar[dr]&&&\\
&H_1(P_\bullet )\ar[ur]&&0&&\\
0\ar[ur]&&&&&\\
}
$$
implies that $H_i (P_\bullet)=\Tor_{i-1}^R (C_1,\cdots ,C_n)$ for $i\geq 2$, $H_1 (P_\bullet)=\ker (C_1\otimes\cdots\otimes C_n \rightarrow F_0^{M_1}\otimes\cdots \otimes F_0^{M_n} )$ and $H_0 (P_\bullet)=\coker (C_1\otimes\cdots\otimes C_n \rightarrow F_0^{M_1}\otimes\cdots \otimes F_0^{M_n} )$.\medskip

\end{rem}

For an $R$-module $M$, set 
$$\TT_{p,q}(M) :=\bigoplus_{i_1<\cdots <i_{p}}\Tor_q^R(M,R/I_{i_1},\ldots ,R/I_{i_{p}})\quad\mbox{and}\quad\HH_{p,q}(M)=\bigoplus_{i_1<\cdots<i_p}H_{q+1}({^+F}_\bullet^{\intF_{i_1,\ldots,i_p}} \otimes_R M [p-1]).$$
In particular, $\TT_{p,0}(M)\simeq M\otimes \SSS^p$ and $\HH_{p,-1}(M)\simeq M\otimes\PPP_p$, for $p>0$. Then Proposition \ref{augmentedmainspectral} provides our first Mayer-Vietoris-type spectral sequence 
\begin{equation}\label{spectralfromsumtoproduct}E^1_{n-p,q}=\TT_{p,q}(M)\Rightarrow \HH_{n,q-p}(M).\end{equation}
The first page of $E$ has the following shape (dotted arrows are differentials on the second page):
$$
\xymatrix@!C @C=0pt @!R @R=10pt{
&&&\TT_{1,2}(M) \ar[r]&\TT_{2,2}(M) \ar[r]&\TT_{3,2}(M) \ar[r]&\cdots\\
 \textcolor{red}{(M,n)}&& &\TT_{1,1}(M) \ar@{-->}[urr] \ar@{..}[ur]\ar[r]&\TT_{2,1}(M)\ar@{..}[ur] \ar[r]&\TT_{3,1}(M) \ar[r]&\cdots\\
& & &M\otimes \SSS^1\ar@{..}[ur] \ar@{-->}[urr] \ar[r]&M\otimes \SSS^2\ar@{..}[ur] \ar[r]&M\otimes \SSS^3 \ar[r]&\cdots \\
 \textcolor{blue}{\HH_{n,1}(M)}\ar@{..}[uuurrr]& \textcolor{blue}{\HH_{n,0}(M)}\ar@{..}[uurr]&  \textcolor{blue}{M/I_1\cdots I_nM}\ar@{..}[ur]&\textcolor{blue}{0}\ar@{..}[ur]&\textcolor{blue}{0}\ar@{..}[ur]&\cdots\\}
$$

In blue: total homology corresponding to the dotted diagonal.

We display this spectral sequence for $M=R$ and $n=2,5$.

\textcolor{red}{(R,2)}:
$$
E^1: \xymatrix@!C@C=0pt@!R@R=10pt{
&&&0&\Tor^R_3(R/I_1,R/I_2) \\
&&&0&\Tor^R_2(R/I_1,R/I_2) \\
  && &0\ar@{..}[ur]&\Tor^R_1(R/I_1,R/I_2)\\
&&&R/I_1\oplus R/I_2\ar@{..}[ur]\ar^{\psi}[r]&R/(I_1+I_2)\\
 \textcolor{blue}{\HH_{2,1}(R)}\ar@{..}[uurururr]& \textcolor{blue}{\HH_{2,0}(R)}\ar@{..}[uurr]& \textcolor{blue}{R/I_1 I_2}\ar@{..}[ur]&\textcolor{blue}{0}\ar@{..}[ur]&\textcolor{blue}{0}\\}
$$

$$
E^2: \xymatrix@!C@C=0pt@!R@R=10pt{
&&&0&\Tor^R_3(R/I_1,R/I_2) \\
&&&0&\Tor^R_2(R/I_1,R/I_2) \\
&&&0\ar@{..}[ur]&\Tor^R_1(R/I_1,R/I_2)\\
&&&R/(I_1\cap I_2)\ar@{..}[ur]&0\\
 \textcolor{blue}{\HH_{2,1}(R)}\ar@{..}[uuuurrrr]& \textcolor{blue}{\HH_{2,0}(R)}\ar@{..}[uurr]& \textcolor{blue}{R/I_1 I_2}\ar@{..}[ur]&\textcolor{blue}{0}\ar@{..}[ur]&\textcolor{blue}{0}\\}
$$
which gives, since $\ker (\psi )\simeq R/(I_1\cap I_2)$, an exact sequence
$$
\xymatrix@C=20pt@R=20pt{
0\ar[r]&\Tor^R_1(R/I_1,R/I_2)\ar[r]&R/I_1 I_2\ar[r]&R/(I_1\cap I_2)\ar[r]&0,\\}
$$
showing that $\Tor^R_1(R/I_1,R/I_2)=(I_1\cap I_2)/I_1 I_2$ and for $i\geq 2$,
$$
H_i (F_\bullet )=\Tor^R_{i}(R/I_1,R/I_2)=H_{i}({^+F}_\bullet^\intF )=\HH_{2,i-2}(R).
$$
Furthermore, by Remark \ref{extensiontomodulesrmk}, $\HH_{2,0}(R)\simeq \ker (I_1\otimes_R I_2\rightarrow I_1 I_2)$ and $\HH_{2,i}(R)\simeq \Tor^R_i (I_1,I_2)$ for $i>0$.\\

\textcolor{red}{(R,5)}: If $\TT_{p,q}=\bigoplus_{i_1<\cdots <i_{p}}\Tor_q^R(R/I_{i_1},\ldots ,R/I_{i_{p}})=0$ for $q>0$ and $p<5$, setting $\TT_q :=\TT_{5,q}=\Tor_{q}^R(R/I_1,\ldots , R/I_5)$ the first page is (underlined zeros are the ones given by the vanishing hypothesis on the modules $\TT_{p,q}$)
$$
\xymatrix@!C@C=10pt@!R@R=10pt{
&0&0&0&\underline{0}&\underline{0}&\underline{0}&\TT_{6}&0\\
&0&0&0&\underline{0}&\underline{0}&\underline{0}&\TT_{5}&0\\
&0&0&0&\underline{0}&\underline{0}&\underline{0}&\TT_{4}&0\\
&0&0&0&\underline{0}&\underline{0}&\underline{0}&\TT_{3}&0\\
&0&0&0&\underline{0}&\underline{0}&\underline{0}&\TT_{2}&0\\
&0&0&0&\underline{0}&\underline{0}&\underline{0}&\TT_{1}&0 \\
&0&0\ar[r]&\SSS^1\ar@{..}[uuuurrrr] \ar@{-->}[uuurrrr] \ar_{\psi_1}[r]&\SSS^2\ar@{..}[uuurrr] \ar_{\psi_2}[r] \ar@{-->}[uurrr]&\SSS^3\ar@{..}[uurr]\ar@{-->}[urr]\ar_{\psi_3}[r]&\SSS^4\ar_{\psi_4}[r]&\SSS^5\ar[r]&0\\
 \textcolor{blue}{\HH_{5,1}}\ar@{..}[uuuuuuurrrrrrr]& \textcolor{blue}{\HH_{5,0}}\ar@{..}[uuuuuurrrrrr]& \textcolor{blue}{\PPP_{5}}\ar@{..}[ur]&\textcolor{blue}{0}\ar@{..}[ur]&\textcolor{blue}{0}\ar@{..}[ur]&\textcolor{blue}{0}\ar@{..}[uurr]&\textcolor{blue}{0}&\textcolor{blue}{0}&\textcolor{blue}{0}\\}
$$
with $\HH_{5,0}=\ker (I_1\otimes\cdots\otimes I_5 \rightarrow I_1\cdots I_5)$ and $\HH_{5,i}=\Tor_{i}^R(I_1,\ldots , I_5)$ for $i>0$, according to Remark \ref{extensiontomodulesrmk}. Maps with dotted arrows correspond to the three unique possibly non zero maps present in the second, third and fourth pages of this spectral sequence.

Although the following definition applies to $R$-modules, we focus only on ideals.

\begin{defn}\label{torinddef}
The ideals $I_1,\ldots, I_n$ are 
\begin{enumerate}[\rm(1)]
    \item Tor-independent if $\Tor_i^R(R/I_1,\cdots,R/I_n)=0$ for all $i>0$.

    
    \item Strongly Tor-independent if any subset of $I_1,\ldots, I_n$ is Tor-independent. 
    
\end{enumerate}
\end{defn}

In other words,  $I_1,\ldots, I_n$ are Tor-independent if and only if the multicomplex $F_\bb=F_\bullet^{1}\otimes\cdots\otimes F_\bullet^{n}$ resolves $R/I_1+\cdots+I_n$.

\begin{rem}\label{equivalencevanishing}
The following are equivalent:
\begin{enumerate}[\rm(i)]
\item $I_1,\ldots, I_n$ are strongly Tor-independent.

\item $I_{j_1}$ and $I_{j_2}+\cdots+I_{j_p}$ are Tor-independent for any $j_1>\cdots > j_p$.
\end{enumerate}
This is easily seen by recursion on $p$.
\end{rem}


We will denote by $\SSS^\bullet_-$ the truncation at degree $1$ of the complex $\SSS^\bullet$; note that such a truncation is the horizontal line $E^1_{\bullet,0}$ of the spectral sequence (\ref{spectralfromsumtoproduct}).

For any $t\geq0$, we introduce the following condition:
$$V_t:\ \mbox{for any}\ 1<p<n\ \mbox{and}\ 0<q<p+t, \TT_{p,q}(R)=0.$$
Note that every strict subset of $I_1,\ldots, I_n$ is Tor-independent if, and only if, the condition $V_t$ is satisfied for all $t\geq0$.

\begin{thm}\label{conditionV}
If $V_0$ is satisfied,
then
\begin{enumerate}[\rm(1)]
\item $H^i(\SSS^\bullet)\simeq\Tor^R_{n-i-1}(R/I_1,\cdots, R/I_n)$ for all $i\geq2$.

\item There exists an exact sequence
$$\xymatrix{\Tor_{n-1}^R(R/I_1,\cdots, R/I_n)\ar[r] & \SSS^0\ar[r] & H^1(\SSS^\bullet_-)\ar[r] & \Tor_{n-2}^R(R/I_1,\cdots, R/I_n)\ar[r] & 0}$$
where the leftmost map is injective if $V_1$ is satisfied.
\item Given $s\geq0$, if $V_{s+1}$ is satisfied, then there exists a natural surjective map $$\Tor^R_{n+s}(R/I_1,\cdots, R/I_n)\to\HHH_{n,s}(R)$$ which is injective if further $V_{s+2}$ is satisfied. In particular, the natural map $$\Tor_n^R(R/I_1,\ldots, R/I_n)\to\ker(I_1\otimes\cdots\otimes I_n\to I_1\cdots I_n)$$ is surjective if $V_1$ is satisfied and an isomorphism if further $V_2$ is satisfied.
\end{enumerate}
\end{thm}

\begin{proof}
By setting $\TT_q:=\TT_{n,q}(R)$, $V_0$ is equivalent to the condition $\TT_{p,q}(R)=0$ for $0<q<p<n$. Thus the spectral sequence (\ref{spectralfromsumtoproduct}) has the following shape if condition $V_0$ holds:

$$\xymatrix@!C@C=10pt@!R@R=5pt{
\ast & \ast & \ast & \ast & \ast & \TT_{n-1}
\\
\ast & \ast & \ast & \ast & 0 & \TT_{n-2}
\\
\ast & \ast & \ast & 0 & 0 & \vdots
\\
\ast & \ast & 0 & 0 & 0 & \vdots
\\
\ast & 0 & 0 & 0 & 0 & \TT_{1}
\\
\SSS^1\ar[r]\ar@{-->}[rrrrruuuu]\ar@{..}[rrrrruuuuu] & \SSS^2\ar[r] & \cdots\ar[r] & \SSS^{n-2}\ar[r]\ar@{-->}[rru] &  \SSS^{n-1}\ar[r] & \SSS^n
}$$
This shows items (1) and (2) since the abutment is $0$ except on the diagonal corresponding to where the terms $\SSS^1$ and $\TT_{n-1}$ appear. Under the conditions $V_{s+1}$ in item (3), the spectral sequence (\ref{spectralfromsumtoproduct}) has the following shape:

$$
\xymatrix@!C @C=3.5pt @!R @R=3.5pt{
\ast&\ast&\ast&\ast&\ast&\ast&\ast&\circ\ar[r]&\TT_{n+s}\\
\ast&\ast&\ast&\ast&\ast&\ast&\circ\ar@{-->}[rru]&0&\vdots\\
\ast&\ast&\ast&\ast&\ast&\circ\ar@{..>}[rrruu]&0&0&\TT_{n}\\
\ast&\ast&\ast&\ast&\circ&0&0&0&\TT_{n-1}\\
\ast&\ast&\ast&\circ&0&0&0&0&\TT_{n-2}\\
\ast&\ast&\circ&0&0&0&0&0&\vdots\\
\ast&\circ&0&0&0&0&0&0&\TT_1 \\
\circ&0&0&0&\SSS^1\ar@{..}[uuuurrrr] \ar@{-->}[uuurrrr] \ar[r]&\SSS^2\ar@{..}[uuurrr] \ar[r] \ar@{-->}[uurrr]\ar[r]&\cdots\ar@{..}[uurr]\ar[r] \ar[r]&\SSS^{n-1}\ar[r]&\SSS^n\\
 \textcolor{blue}{\HH_{s}}\ar@{..}[uuuuuuuurrrrrrrr]& \textcolor{blue}{\cdots}&\textcolor{blue}{\HH_{0}}\ar@{..}[uuuuuurrrrrr]& \textcolor{blue}{\PPP_{n}}\ar@{..}[ur]&\textcolor{blue}{0}\ar@{..}[ur]&\textcolor{blue}{0}\ar@{..}[ur]&\textcolor{blue}{0}\ar@{..}[uurr]&\textcolor{blue}{0}&\textcolor{blue}{0}\\}
$$
that provides the natural surjection $\TT_{n+s}\to\HH_{s}:=\HH_{n,s}(R)$. If $V_{s+2}$ is further satisfied, the modules placed at circles $\circ$ 
are zero, showing that this map is indeed an isomorphism.

\end{proof}

\begin{cor}\label{homologyofSSS}
If any strict subset of $I_1,\ldots, I_n$ is Tor-independent, then
\begin{enumerate}[\rm(1)]
\item $H^i(\SSS^\bullet)\simeq\Tor^R_{n-i-1}(R/I_1,\cdots, R/I_n)$ for all $i\geq2$.

\item There exists an exact sequence
$$\xymatrix{0\ar[r] & \Tor_{n-1}^R(R/I_1,\cdots, R/I_n)\ar[r] & \SSS^0\ar[r] & H^1(\SSS^\bullet_-)\ar[r] & \Tor_{n-2}^R(R/I_1,\cdots, R/I_n)\ar[r] & 0.}$$

\item $\Tor^R_{n+i}(R/I_1,\cdots, R/I_n)\simeq\HHH_{n,i}(R)$ for all $i\geq0$. In particular, $$\Tor^R_n(R/I_1,\cdots, R/I_n)\simeq\ker(I_1\otimes\cdots\otimes I_n\to I_1\cdots I_n).$$
\end{enumerate}
\end{cor}

\begin{cor}\label{torind}
Let $F_\bullet^{j}$, with $F_0^{j}=R$, be a free resolution of $R/I_j$ for $1\leq j\leq n$, then $F_\bb :=F_\bullet^{1}\otimes\cdots\otimes F_\bullet^{n}$ is a $n$-multicomplex such that 
\begin{enumerate}[\rm(1)]
    \item ${^+F}_\bullet^\intF$ resolves $R/I_1\cdots I_n$ if $I_j$ and $I_1\cdots I_{j-1}$ are Tor-independent for $1<j\leq n$.


 \item If $I_1,\ldots, I_n$ are strongly Tor-independent, then $I_j$ and $I_1\cdots I_{j-1}$ are Tor-independent for $1<j\leq n$.
Hence  $F_\bullet$ resolves $R/(I_1+\cdots +I_n)$,  ${^+F}_\bullet^\intF$ resolves $R/I_1\cdots I_n$ and the complex 
$$
\SSS^\bullet:\xymatrix{
0\ar[r]&\SSS^0\ar[r]&\SSS^{1}\ar[r]&\SSS^{2}\ar[r]&\cdots\ar[r]&\SSS^n \ar[r]&0\\}
$$
 is exact.
\end{enumerate}
\end{cor}

\begin{proof}
For $n=2$, (1) and (2) are consequences of the spectral sequence $ \textcolor{red}{(R,2)}$.\\
Item (1) then follows by induction on $n\geq 2$, and taking into account that ${^+F}_\bullet^{\intF}={^+}(F_\bullet^{1}\otimes{^+G}_\bullet^{\intF})^{\intF}$ with $G_\bullet=F_\bullet^{2}\otimes\cdots\otimes F_\bullet^{n}$.

We now prove (2). It is trivial that $F_\bullet$ resolves $R/(I_1+\cdots+I_n)$. Corollary \ref{homologyofSSS} assures that ${^+F}_\bullet^{\intF}$ resolves $R/I_1\cdots I_n$ and that $\SSS^\bullet$ is exact.

It remains to prove that $\Tor_{i}^R(R/I_{j},R/I_1\cdots I_{j-1})=0$ for $i>0$. We induct on $j\geq 2$. 

Consider the complex $\SSS^\bullet$ for the ideals $I_1,\ldots, I_j$:
$$
\xymatrix{
0\ar[r]&R/I_1\cdots I_j\ar[r]&\SSS^{1}\ar[r]&\SSS^{2}\ar[r]&\cdots\ar[r]&\SSS^j \ar[r]&0.\\}
$$
The strongly Tor-independence assures that it is exact and the fact that $\Tor_i^R(R/I_{j},R/I_{j_1}+\cdots +I_{j_p})=0$ for $j\geq j_1>\cdots > j_p\geq 1$ and $i>0$ implies that $\Tor_i^R(R/I_{j},\SSS^p)=0$ for $j\geq p\geq 1$ and $i>0$. 
This provides a spectral sequence abutting to zero with $E^1$ page
$$
\xymatrix{
\Tor_{2}^R(R/I_{j},R/I_1\cdots I_{j-1})&0&\cdots&0&0\\
\Tor_{1}^R(R/I_{j},R/I_1\cdots I_{j-1})&0&\cdots&0&0\\
R/(I_1\cdots I_{j-1} +I_{j})\ar[r]&\SSS^{1}/I_{j}\SSS^{1}\ar[r]&\cdots\ar[r]&\SSS^{j-1}/I_{j}\SSS^{j-1}\ar[r]&\SSS^{j}\\}
$$
 It follows that $\Tor_{i}^R(R/I_{j},R/I_1\cdots I_{j-1})=0$ for $i>0$ and the bottom line is exact. The exactness of the bottom line also follows from the spectral sequence (\ref{spectralfromsumtoproduct}) for $ \textcolor{red}{(R/I_{j},j-1)}$.
 \end{proof}

\begin{cor}\label{spectraltorfromsumtoproduct}
If $I_1,\ldots, I_n$ are strongly Tor-independent, then for any $R$-module $M$ there exists a spectral sequence
$$E^1_{p,q}=\oplus_{i_1<\cdots<i_p}\Tor^R_q(M, R/I_{i_1}+\cdots+I_{i_p})\Rightarrow\Tor^R_{q-p+1}(M, R/I_1\cdots I_n).$$
\end{cor}

\begin{proof}
Consider the two spectral sequences arising from the tensor product of $\SSS^\bullet_-$ with a resolution of $M$ and then apply Corollary \ref{torind}(2).
\end{proof}

\subsubsection{The spectral sequence Proposition \ref{fromaugmentedinteriorspectral}}

For an $R$-module $M$, Proposition \ref{fromaugmentedinteriorspectral} provides another Mayer-Vietoris-type spectral sequence 
\begin{equation}\label{spectralfromproducttosum}
E^1_{p,q}=\HH_{p,q}(M)\Rightarrow\TT_{n,p+q}(M).\end{equation}
The first page of $E$ has the following shape (dotted arrows are differentials on the second page):
$$
\xymatrix@C=20pt@R=20pt{
&\cdots\ar[r]&\HH_{3,1}(M) \ar[r]&\HH_{2,1}(M) \ar[r]&\HH_{1,1}(M)\\
\textcolor{red}{(M,n)} &\cdots\ar[r] &\HH_{3,0}(M) \ar@{-->}[urr]\ar[r]&\HH_{2,0}(M)\ar@{..}[ur] \ar[r]&\HH_{1,0}(M)\\
 &\cdots\ar[r] & M\otimes \PPP_3\ar@{..}[ur] \ar@{-->}[urr] \ar[r]&M\otimes \PPP_2\ar@{..}[ur] \ar[r]&M\otimes \PPP_1\\
\textcolor{blue}{\TT_{n,3}(M)}\ar@{..}[uuurrr]& \textcolor{blue}{\TT_{n,2}(M)}\ar@{..}[uuurrr]& \textcolor{blue}{\TT_{n,1}(M)}\ar@{..}[uurr]&  \textcolor{blue}{M\otimes\SSS^n}\ar@{..}[ur]}
$$

In blue: homology corresponding to the dotted diagonal.

As done previously in the spectral sequence (\ref{spectralfromsumtoproduct}), we display this one for $M=R$ and $n=2,4$.

\textcolor{red}{(R,2)}:
$$
E^1: \xymatrix@C=20pt@R=20pt{
&&&\HHH_{2,1}(R)&0\\
&&&\HHH_{2,0}(R)&0\\
&&&R/I_1 I_2\ar[r]&R/I_1\oplus R/I_2&\\\
 \ar@{..}[ururur]&\textcolor{blue}{\Tor_2^R(R/I_1,R/I_2)}\ar@{..}[uuurrr]& \textcolor{blue}{\Tor_1^R(R/I_1,R/I_2)}\ar@{..}[uurr]&  \textcolor{blue}{R/(I_1 +I_2)}\ar@{..}[ur]\\}
$$
which gives the exact sequence
$$\xymatrix@C=20pt@R=20pt{0\ar[r] & \Tor_1^R(R/I_1,R/I_2)\ar[r] & R/I_1I_2\ar[r] & R/I_1\oplus R/I_2\ar[r] & R/I_1+I_2\ar[r] & 0}$$
and $\Tor_i^R(R/I_1,R/I_2)=\HHH_{2,i-2}(R)=H_i({^+F}_\bullet^\intF)$ for all $i\geq2$.



\textcolor{red}{(R,4)}:
$$
E^1: \xymatrix@C=10pt@R=15pt{
&&&&\HHH_{4,1}(R)\ar[r]&\HHH_{3,1}(R)\ar[r]&\HHH_{2,1}(R)&0\\
&&&&\HHH_{4,0}(R)\ar@{-->}[urr]\ar[r]&\HHH_{3,0}(R)\ar[r]&\HHH_{2,0}(R)&0\\
&&&&R/I_1I_2I_3I_4\ar@{-->}[urr]\ar[r]&\oplus_{i<j<k}R/I_i I_j I_k\ar[r]&\oplus_{i<j} R/I_i I_j\ar[r]&\oplus_i R/I_i&\\
 &&&\textcolor{blue}{\TT_3}\ar@{..}[ururur]&\textcolor{blue}{\TT_2}\ar@{..}[uururr]& \textcolor{blue}{\TT_1}\ar@{..}[uurr]&  \textcolor{blue}{R/(I_1 +I_2+I_3+I_4)}\ar@{..}[ur]\\}
$$
where $\TT_\ell=\Tor_\ell^R(R/I_1,R/I_2,R/I_3,R/I_4)$.

If $\HH_{3,q}(R)=\HH_{2,q}(R)=0$ for all $q\geq0$,
$$
E^1: \xymatrix@C=10pt@R=15pt{
&&&&\HHH_{4,1}(R)&0&0&0\\
&&&&\HHH_{4,0}(R)&0&0&0\\
&&&&R/I_1I_2I_3I_4\ar[r]&\oplus_{i<j<k}R/I_i I_j I_k\ar[r]&\oplus_{i<j} R/I_i I_j\ar[r]&\oplus_i R/I_i&\\
 &&&\textcolor{blue}{\TT_3}\ar@{..}[ururur]&\textcolor{blue}{\TT_2}\ar@{..}[uururr]& \textcolor{blue}{\TT_1}\ar@{..}[uurr]&  \textcolor{blue}{R/(I_1 +I_2+I_3+I_4)}\ar@{..}[ur]\\}
$$
It provides a complex
$$\xymatrix@=1em{0\ar[r]&R/I_1I_2I_3I_4\ar[r]^-{\alpha}&\oplus_{i<j<k}R/I_i I_j I_k\ar[r]^\beta&\oplus_{i<j} R/I_i I_j\ar[r]^\gamma&\oplus_i R/I_i\ar[r]^-{\theta} & R/(I_1+I_2+I_3+I_4)\ar[r] & 0}$$
with homologies
$$\TT_3=\ker\alpha,\ \TT_2=\ker\beta/\ima\alpha,\ \TT_1=\ker\gamma/\ima\beta$$
and for $i\geq4$,
$$\TT_i=\HH_{4,i-4}(R)=H_i({^+F}^\intF_\bullet).$$

\begin{prop}\label{homologyofPPP}
If any strict subset of $I_1,\ldots, I_n$ is Tor-independent, then
\begin{enumerate}[\rm(1)]
\item $H_i(\PPP_\bullet)\simeq\Tor^R_{i-1}(R/I_1,\ldots, R/I_n)$ for all $i\leq n$.

\item $\Tor^R_{n+i}(R/I_1,\ldots, R/I_n)\simeq\HH_{n,i}(R)$ for all $i\geq0$.
\end{enumerate}

In particular, if $I_1,\ldots, I_n$ are strongly Tor-independent, then the complex 
$$
\xymatrix{
0\ar[r]&\PPP_{n}\ar[r]&\cdots\ar[r]&\PPP_2\ar[r]&\PPP_1\ar[r]&R/(I_1+\cdots+I_n)\ar[r] & 0\\}
$$
is exact.
\end{prop}

\begin{proof}
Tor-independence of strict subsets of the $I_1,\ldots, I_n$ is equivalent to strongly Tor-independence of $I_{i_1},\ldots, I_{i_p}$ for any $1\leq i_1<\cdots<i_p\leq n$ with $p<n$. Hence, by Corollary \ref{torind} $\HH_{p,q}=0$ for $p\neq n$ and $q\neq-1$. The spectral sequence (\ref{spectralfromproducttosum}) then has the following shape

$$\xymatrix{
\HH_{n,1} & 0 & 0 & \cdots & 0 & 0 & 0
\\
\HH_{n,0} & 0 & 0 & \cdots & 0 & 0 & 0
\\
\PPP_n\ar[r] & \PPP_{n-1}\ar[r] & \PPP_{n-2}\ar[r] & \cdots\ar[r] & \PPP_3\ar[r] & \PPP_2\ar[r] & \PPP_1
}$$
The result follows by convergence.
\end{proof}

\begin{rem}\label{descriptionofmultitors}
Let $p<n$. If $I_{i_1},\ldots, I_{i_q}$ are Tor-independent for all $q\leq p$ and $i_1<\cdots<i_q$, then
$$\Tor_{i}^R(R/I_1,\ldots, R/I_n)\simeq H_{i+1}(\PPP_\bullet), \forall 1\leq i\leq p.$$
It follows along the same lines as in Proposition \ref{homologyofPPP} since $\HH_{q,i}(R)=0$ for every $q\leq p$ (Proposition \ref{homologyofPPP}(1) is the case $p=n-1$).
\end{rem}

The homologies of the complexes $\SSS^\bullet$ and $\PPP_\bullet$ are completely related under Tor-independence.

\begin{cor}\label{torind2}
If any strict subset of $I_1,\ldots, I_n$ is Tor-independent then
\begin{enumerate}[\rm(1)]
\item $H_i(\PPP_\bullet)\simeq H^{n-i}(\SSS^\bullet)$ for all $i\leq n-2$.

\item There exists an exact sequence
$$\xymatrix{0\ar[r] & H_n(\PPP_\bullet)\ar[r] & \SSS^0\ar[r] & H^1(\SSS^\bullet_-)\ar[r] & H_{n-1}(\PPP_\bullet)\ar[r] & 0}$$
where $\SSS^\bullet_-$ is $\SSS^\bullet$ truncated at degree $1$.
\end{enumerate}
\end{cor}

\begin{proof}
Apply corollaries \ref{homologyofSSS} and \ref{homologyofPPP}.
\end{proof}

Similar to Corollary \ref{spectraltorfromsumtoproduct}, we have the following spectral sequence:

\begin{cor}\label{spectraltorfromproducttosum}
If $I_1,\ldots, I_n$ are strongly Tor-independent, then there exists a spectral sequence
$$E^1_{p,q}=\oplus_{i_1<\cdots<i_p}\Tor^R_q(M, R/I_{i_1}\cdots I_{i_p})\Rightarrow\Tor^R_{p+q-1}(M, R/I_1+\cdots+I_n).$$
\end{cor}

\begin{rem}
In general, for any ideals $I_1,\ldots, I_n$, choosing a flat resolution $F_\bullet$ of an $R$-module $M$, the double complex $\PPP_\bullet\otimes_RF_\bullet$ gives rise to two spectral sequences
$$E^1_{p,q}=\oplus_{i_1<\cdots<i_p}\Tor^R_q(M, R/I_{i_1}\cdots I_{i_p})\quad\mbox{and}\quad {}'E^2_{p,q}=\Tor_q^R(M,H_p(\PPP_\bullet))$$
abutting to two filtrations of the total homology. The spectral sequence ${}'E^2$ degenerates when the ideals are strongly Tor-independent, which is the case of Corollary \ref{spectraltorfromproducttosum}. Similarly, the double complex $\SSS^\bullet_-\otimes_RF_\bullet$ gives rise to two spectral sequences
$$E^1_{-p,q}=\oplus_{i_1<\cdots<i_p}\Tor^R_q(M, R/I_{i_1}+\cdots+I_{i_p})\quad\mbox{and}\quad {}'E^2_{-p,q}=\Tor_q^R(M,H^p(\SSS^\bullet_-))$$
abutting to two filtrations of the total homology. Again, $'E^2$ degenerates when the ideals are strongly Tor-independent, as in Corollary \ref{spectraltorfromsumtoproduct}.

\end{rem}


\subsection{Equivalences of exactness}

The spectral sequences (\ref{spectralfromsumtoproduct}) and (\ref{spectralfromproducttosum}) allow us to relate Tor-independence and exactness of the complexes $\PPP_\bullet$ and $\SSS^\bullet_-$ (items (3) and (4) below, respectively). Given $1\leq l_1<\cdots<l_t\leq n$, we denote by $\HH^{l_1,\ldots,l_t}_{p,q}(R)$ and $\TT^{l_1,\ldots,l_t}_{p,q}(R)$ the corresponding modules in the spectral sequences (\ref{spectralfromsumtoproduct}) and (\ref{spectralfromproducttosum}) when one considers only the subset $I_{l_1},\ldots, I_{l_t}$ of the $I_1,\ldots, I_n$.

\begin{prop}\label{equivalencesofexactness}
Consider the following properties:

\begin{enumerate}[\rm(1)]
\item $I_1,\ldots, I_n$ are strongly Tor-independent.

\item $\HH_{\bullet,q}^{l_1,\ldots,l_t}(R)$ is exact for any $1\leq l_1<\cdots<l_t\leq n$ and $q\geq0$.

\item $\HH_{\bullet,-1}^{l_1,\ldots,l_t}(R)$ is exact for any $1\leq l_1<\cdots<l_t\leq n$.

\item $\TT_{\bullet,0}^{l_1,\ldots,l_t}(R)$ is exact for any $1\leq l_1<\cdots<l_t\leq n$.
\end{enumerate}

Then $(1)\Leftrightarrow(2)+(3)\Leftrightarrow(2)+(4)$.
\end{prop}

\begin{proof}
It is imediate to see that $(1)\Rightarrow(2)+(4)$ follows from Corollary \ref{torind}, $(1)\Rightarrow(2)+(3)$ follows from Corollary \ref{torind} and Proposition \ref{homologyofPPP}, and that $(2)+(3)\Rightarrow(1)$ follows from the spectral sequence (\ref{spectralfromproducttosum}). To prove that $(2)+(4)\Rightarrow(1)$, we induct on $t$ to assume that any strict subsets of the $I_1,\ldots, I_n$ are Tor-independent, so apply Corollary \ref{torind2} and the convergence of the spectral sequence (\ref{spectralfromproducttosum}).
\end{proof}

\begin{rem}
It follows from Theorem \ref{rigtor}(1)(b) that, for proper ideals over a regular local ring containing a field, all the conditions in the Proposition \ref{equivalencesofexactness} are equivalent to $\Tor^R_1(R/I_1,\ldots, R/I_n)=0$.
\end{rem}

\begin{rem}\label{torinddefformod}
All the results above could be extended, again by propositions \ref{augmentedmainspectral} and \ref{fromaugmentedinteriorspectral}, to non necessarily cyclic $R$-modules. Indeed, given $R$-modules $M_1,\ldots, M_n$ and corresponding free resolutions $F^{M_1}_\bullet,\ldots, F^{M_n}_\bullet$, we set
$$\SSS^p:=\bigoplus_{i_1<\ldots<i_p}M_{i_1}\otimes\cdots\otimes M_{i_p}\quad\mbox{and}\quad\PPP_p:=\bigoplus_{i_1<\ldots<i_p}H_0({}^+(F_\bullet^{M_{i_1}}\otimes\cdots\otimes F_\bullet^{M_{i_p}})^\intF[p-1])$$
for $p>0$. The complexes $\SSS^\bullet$ and $\PPP_\bullet$ are well defined (up to isomorphism, according to Remark \ref{extensiontomodulesrmk}) via the Mayer-Vietoris spectral sequences (\ref{spectralfromsumtoproduct}) and (\ref{spectralfromproducttosum}). All the results above could then be extended along the same lines of proof.
\end{rem}

\subsection{Multigraded Tor modules}

We close the paper by relating the non-vanishing region of Tor modules of $R$-modules $M$ against products and sums. Let $S$ be a commutative unitary ring,  $G$ an abelian group, and $R=S[X_{1},\ldots, X_{n}]$ with a $G$-grading. By a graded $ R$-module, we mean an $R$-module with a $G$-grading.



\begin{defn}\label{supdef}
The support of a graded $R$-module $M$ is $$\Supp_{G}(M):=\{\gamma\in G :M_\gamma\neq0\}.$$

Given a graded $R$-module $M$ and a homogeneous ideal $I$ of $R$,
$$
\mathbb T_j^{I} (M):=\Supp_{G} (\Tor_j^R(M,R/I))$$
and ${\mathbb T}^{I}(M):=\cup_j \mathbb T_j^{I} (M)$. 

\end{defn}

Let $J\subseteq[n]=\{1,\ldots,n\}$. For each set of variables $\mathbf{X}_J:=\{X_{i},\ i\in J\}$, write $B_J$ for the ideal generated by these variables. 

\begin{prop}\label{supportoftors}
Let $M$ be a  graded $R$-module. If $J_1,\ldots ,J_s$ are non empty subsets  of $[n]$ such that $J_i\cap J_j=\emptyset$ for any $i\neq j$, then $$\cup_{i_1<\cdots<i_p}{\mathbb T}^{B_{J_{i_1}}\cdots B_{J_{i_p}}}(M)=\cup_{i_1<\cdots<i_p}{\mathbb T}^{B_{J_{i_1}}+\cdots+B_{J_{i_p}}}(M).$$

\end{prop}

\begin{proof}
For any $1\leq i_1<\cdots<i_p\leq s$, The Koszul complex $K_\bullet(\mathbf{X}_{J_{i_1}};R)\otimes_R\cdots\otimes_RK_\bullet(\mathbf{X}_{J_{i_p}};R)\simeq K_\bullet(\mathbf{X}_{J_{i_1}}\cup\cdots\cup\mathbf{X}_{J_{i_p}};R)$ is the minimal free resolution of $R/B_{J_{i_1}\cup\cdots\cup J_{i_p}}$ which precisely means that the ideals $B_{J_1},\ldots,B_{J_s}$ are strongly Tor-independent. Given $i_1<\cdots<i_u$, from the spectral sequence in Corollary \ref{spectraltorfromsumtoproduct}
$$
E_{p,q+p-1}^1=\oplus_{i_1\leq j_1<\cdots <j_{p}\leq i_u}\Tor^R_{q+p-1}(M,R/B_{j_1}+\cdots +B_{j_{p}})\Rightarrow \Tor_{q}^R (M,R/B_{i_1}\cdots B_{i_u})
$$
we get $${\mathbb T}^{B_{i_1}\cdots B_{i_u}}(M)\subseteq\cup_{i_1\leq j_1<\ldots<j_p\leq i_u}{\mathbb T}^{B_{i_1}+\cdots+B_{i_p}}(M).$$

On the other hand, for each $i_1<\cdots<i_u$ the spectral sequence in Corollary \ref{spectraltorfromproducttosum}
$$E^1_{p,q-p+1}=\oplus_{i_1\leq j_1<\ldots<j_p\leq i_u}\Tor_{q-p+1}^R(M, R/B_{j_1}\cdots B_{j_p})\Rightarrow\Tor_{q}^R(M,R/B_{i_1}+\cdots+B_{i_u})$$
assures that ${\mathbb T}^{B_{i_1}+\ldots+B_{i_u}}(M)\subseteq\cup_{i_1\leq j_1<\ldots<j_p\leq i_u}{\mathbb T}^{B_{j_1}\cdots B_{j_p}}(M)$, whence the result.
\end{proof}

\noindent{\bf Acknowledgements.}
The first-named author was partially supported by the IIT Kharagpur start-up research grant, IIT Kharagpur CPDA, the second-named author was supported by the France-Brazil RFBM network, and the third-named author was partially supported by CNPq, grant 408698/2023-3. Part of this work was done when the first author visited Institut de Mathématiques de Jussieu – Paris Rive Gauche and he would like to thank the institute for its hospitality. Part of this paper was written at the Centre International de Rencontres Mathématiques (CIRM). The authors are very appreciative of the hospitality offered by CIRM.


\begin{thebibliography}{9}
\bibliographystyle{alpha}

\bibitem[Aus61]{Aus}
M. Auslander, {\it Modules over unramified regular local rings}, Illinois J. Math. \textbf{5} (1961), 631-647.

\bibitem[BH99]{BH}
W. Bruns, J. Herzog, {\it Cohen-{M}acaulay rings}, Cambridge Stud. in
  Adv. Math. \textbf{39}. Cambridge University Press, Cambridge, 1993.

\bibitem[CE56]{CE}
H. Cartan, S. Eilenberg, {\it Homological algebra}, Princeton University Press, Princeton (1956).

\bibitem[CHN26]{CHN}
M. Chardin, R. Holanda, J. Naéliton, {\it Homology of multiple complexes and Mayer-Vietoris spectral sequences}, Proc. Amer. Math. Soc. \textbf{154} (2026), 1359-1371.


 \bibitem[Gro61]{EGAIII}
 A. Grothendieck, {\it \'El\'ements de g\'eom\'etrie alg\'ebrique. III. \'Etude cohomologique des faisceaux coh\'erents. I.} Inst. Hautes \'Etudes Sci. Publ. Math. No. 11 (1961).

\bibitem[Lic66]{Lic}
S. Lichtenbaum, {\it On the vanishing of Tor in regular local rings}, Illinois J. Math. \textbf{10} (1966), 220-226.

\bibitem[Ser65]{Ser}
J.-P. Serre, {\it Algèbre locale. Multiplicités}, Lecture Notes in Mathematics, vol. 11. Springer-Verlag, Berlin-New York (1965). Cours au Collège de France, 1957-1958, rédigé par Pierre Gabriel, Seconde édition, 1965.


\bibitem[Sch98]{Sch}
P. Schenzel, {\it On the use of local cohomology in algebra and geometry}, In: Six Lectures on Commutative Algebra (Bellaterra, 1996), pp. 241-292. Progr. Math., 166 Birkhäuser, Basel (1998).

\bibitem[Ver96]{Ver}
J.-L. Verdier, {\it Des catégories dérivées des catégories abéliennes}, Astérisque, no. 239 (1996).


\end{thebibliography}
\end{document}